\documentclass[a4paper, 12pt,reqno]{article} 
\usepackage{authblk}

\usepackage{hyperref}
\hypersetup{
    colorlinks,
    linkcolor={blue!30!black},
    citecolor={blue!50!black},
    urlcolor={blue!80!black}
}

\usepackage[utf8]{inputenc} \usepackage[T1]{fontenc} 
\usepackage{natbib} 
\usepackage{bm}
\usepackage{geometry}
\usepackage[normalem]{ulem}
\usepackage{graphicx}%
\usepackage{multirow}%
\usepackage{amsmath,amssymb,amsfonts}%
\usepackage{amsthm}%
\usepackage{mathrsfs}%
\usepackage[title]{appendix}%
\usepackage{xcolor}%
\usepackage{textcomp}%
\usepackage{manyfoot}%
\usepackage{booktabs}%
\usepackage{algorithm}%
\usepackage{algorithmicx}%
\usepackage{algpseudocode}%
\usepackage{listings}%

\theoremstyle{plain}
\newtheorem{theorem}{Theorem}[section]
\newtheorem{lemma}{Lemma}[section]
\newtheorem{proposition}{Proposition}[section]

\newtheorem{remark}{Remark}[section] 

\theoremstyle{definition}
\theoremstyle{definition}
\newtheorem*{discussion*}{Discussion}

\newtheorem{example}{Example}[section]
\newtheorem{condition}{Condition} 
\newtheorem{assumption}{Assumption}

\newtheorem*{remark*}{Remark}
\newtheorem{corollary}{Corollary}[section]

\usepackage{stmaryrd}

\usepackage{dsfont}

\newcommand{\rset}{\mathbb{R}}

\newcommand{\ind}{\mathds{1}}
\newcommand{\abs}[1]{\lvert{#1}\rvert}

\newcommand{\eqd}{\stackrel{d}{=}}
\newcommand{\ffrac}[2]{\ensuremath{\frac{\displaystyle #1}{\displaystyle #2}}}
\newcommand{\un}[1]{\ind{\left\{#1\right\}}}
\newcommand{\indic}{\un} 
\newcommand{\PP}[1][]{\ifthenelse{\equal{#1}{}}{\ensuremath{\mathbb{P}}}{\ensuremath{\mathbb{P}\left( #1 \right) }}}
\newcommand{\EE}[1][]{\ifthenelse{\equal{#1}{}}{\ensuremath{\mathbb E}}{\ensuremath{{\mathbb E}\left[ #1 \right]}}}
\newcommand{\Var}[1][]{\ifthenelse{\equal{#1}{}}{\ensuremath{\mathrm{Var}}}{\ensuremath{{\mathrm{Var}}\left[ #1 \right]}}}
\newcommand{\Cov}[1][]{\ifthenelse{\equal{#1}{}}{\ensuremath{\mathrm{Cov}}}{\ensuremath{{\mathrm{Cov}}\left[ #1 \right]}}}

\DeclareMathOperator*{\argmin}{arg\,min}

\DeclareMathOperator{\range}{range}
\DeclareMathOperator{\sign}{sign}

\usepackage{xspace}

\newcommand\ie{\emph{i.e.}\xspace}
\newcommand\iid{\ensuremath{\mathit{i.i.d.}}\xspace}

\newcommand{\ud}{\,\mathrm{d}}
\newcommand{\point}{\,\cdot\,}
\newcommand{\given}[1][{}]{\;\middle\vert\;{#1} }

\usepackage{xcolor}

\newcommand{\EVT}{\textsc{EVT}\xspace}
\newcommand{\Lasso}{\textsc{LASSO}\xspace}
\newcommand{\ERM}{\textsc{ERM}\xspace}
\newcommand{\CV}{\textsc{CV}\xspace}

\newcommand{\VC}{\textsc{VC}\xspace}
\newcommand{\EVA}{\textsc{EVA}\xspace}
\newcommand{\loo}{\textit{l.o.o.}\xspace}
\newcommand{\lpo}{\textit{l.p.o.}\xspace}
\newcommand\eg{\emph{e.g. }\xspace}

\newcommand{\wrt}{\textit{w.r.t.}\xspace}
\newcommand{\rhs}{\textit{r.h.s.}\xspace}

\newcommand{\TR}{\mathcal{R}}
\newcommand{\ER}[1][]{\widehat{ \mathcal{R}}}
\newcommand{\risk}{\mathcal{R}}
\newcommand{\rcvEx}{\widehat\risk_{\mathrm{CV},\alpha}}
\newcommand{\rCV}{\widehat\risk_{\mathrm{CV}}}
\newcommand{\data}{\mathcal{D}}
\newcommand{\DD}{\data}

\newcommand{\BiasCV}{\mathrm{Bias}}
\newcommand{\DevCV}{\mathrm{D_{cv}}}
\newcommand{\DevExt}{\mathrm{D_{t_\alpha}}}
\newcommand{\vapnikG}{\mathcal{V}_{\mathcal{G}}}
\newcommand{\Kfold}{\textrm{K-fold}\xspace}

\newcommand{\sphere}{S}
\numberwithin{equation}{section}

\begin{document}

\title{Cross-validation on Extreme Regions}

\author[a]{Anass Aghbalou}

\author[b]{Patrice Bertail} 

\author[c]{Fran\c{c}ois Portier}

\author[d]{Anne Sabourin}

\affil[a]{ LTCI, T\'el\'ecom Paris, Institut Polytechnique de Paris, Palaiseau, France}
\affil[b]{Universit\'e Paris Nanterre, MODAL'X, Nanterre, France}
\affil[c]{Ensai, CREST, Rennes, France}
\affil[d]{Universit\'e Paris Cit\'e, CNRS, MAP5, F-75006 Paris, France}




\maketitle

\begin{abstract}
  We conduct a non-asymptotic study of the Cross-Validation
  (CV) estimate of the generalization risk for learning algorithms
  dedicated to extreme regions of the covariates space. 
  In this context which has recently been analysed from an Extreme Value Analysis perspective, 
  the risk function measures the
  algorithm's error given that the norm of the input exceeds a high
  quantile. The main challenge within this framework is the negligible
  size of the extreme training sample with respect to the full sample
  size and the necessity to re-scale the risk function by a
  probability tending to zero. We open the road to a finite sample
  understanding of CV for extreme values by establishing two new
  results: an exponential probability bound on the \Kfold CV error and
  a polynomial probability bound on the leave-\textrm{p}-out CV. Our
  bounds are sharp in the sense that they match state-of-the-art
  guarantees for standard CV estimates while extending them to
  encompass a conditioning event of small probability. We illustrate
  the significance of our results regarding high dimensional
  classification in extreme regions via a Lasso-type logistic
  regression algorithm. The tightness of our bounds is investigated in
  numerical experiments.

{\bf  Keywords:} Extreme value analysis, cross-validation, concentration inequalities. 
\end{abstract}

\section{Introduction}
\label{sec:intro}

Cross-validation (\CV) is a most popular statistical learning tool for
estimating the generalization risk of an algorithm and for
hyper-parameter or model selection
(\cite{arlot2010survey,wager2020cross,bates2021cross}). Although the
performance of \CV\ has been analysed in a variety of
settings 
including density estimation \citep{arlot2008VFpen,arlot2016VFchoice}
or Least-squares regression
\citep{homrighausen2013lasso,xu2020rademacher}, the literature is
silent about theoretical guarantees of \CV\ when applied to algorithms
dedicated to Extreme Value Analysis (\EVA). The goal of this paper is
to make a first step in this direction.  We consider in particular
learning algorithms based on Empirical Risk Minimization (\ERM)
restricted to low probability regions of the covariate  space, 
recently
considered in
\cite{jalalzai2018binary,jalalzai2020heavy,clemencon2022concentration} and \cite{huet2023regression}.


Our study is originally motivated by several works in the \EVA literature in which  \CV comes as a natural  tool to evaluate an estimation error or to select hyper-parameters. 
In an unsupervised context, this includes in particular parametric inference of multivariate tail dependence (\cite{einmahl2012m,einmahl2018continuous,einmahl2016m,kiriliouk2019peaks}), where \CV\ could be naturally envisioned for evaluating  goodness-of-fit or for choosing  between competing models.
In the recent line of works concerned with dimension reduction in Multivariate Extremes (see  the review by \cite{engelke2021sparse} and the references therein)  the question of hyper-parameters and sparsity level selection cannot be avoided in practice. As an example, in \cite{goix2016sparse,goix2017sparse}, the number of subcones of $\rset^d$ supporting the limit tail measure, or equivalently the cut-off level below which the empirical mass is deemed negligible, must be chosen by the user. As noted in Remark~5 in \cite{goix2017sparse} this can be recast as a penalized risk minimization problem. In dimension reduction methods based on Principal Component Analysis \citep{cooley2019decompositions,jiang2020principal,drees2021principal} the dimension reduction space minimizes an empirical reconstruction risk and the dimension of the output must again be chosen by the user. Here CV could be used
to evaluate the reconstruction error, as opposed to the empirical risk computed on the training set which typically yields optimistic estimates, leading to  over-fitting. 
Clustering approaches to dimension reduction constitute still another example of \ERM\ frameworks that have successfully been generalized to  the context of \EVA \citep{janssen2020k,jalalzai2021feature} and where, among others, the number of clusters should be chosen.  
In a supervised context, extreme quantile regression has a long history in \EVT (see \emph{e.g.} \cite{daouia2013kernel,chernozhukov2017extremal,chavez2014extreme} and the references therein). For kernel based methods, choosing the kernel bandwith by \CV comes as a natural idea. Alternative approaches to kernel methods have been recently proposed  based on Gradient Boosting \citep{velthoen2021gradient}, Regression Trees \citep{farkas2021cyber} or Extremal Random Forests \citep{gnecco2022extremal}.
 These
approaches also come with tuning parameters. 
Finally, in the supervised framework of classification  in extreme regions  \citep{jalalzai2018binary,jalalzai2020heavy,clemencon2022concentration}, or in recent  extensions to continuous regression considered in~\cite{huet2023regression},  \CV\ is a natural candidate for estimating the generalization risk of the output or, in a high dimensional setting, for feature selection.

\paragraph{\bf Purpose of this work and leading example.} Our goal is
to open the road to a finite-sample understanding of the guarantees
enjoyed by \CV\ in algorithms dedicated to extreme values.  To fix
ideas, the learning problems we have in mind involve independent and identically distributed (\iid) data
$Z_i, i\le n$ in a sample space $\mathcal{Z}$,  
and a low
probability region $\mathbb{A}\subset \mathcal{Z}$, typically
$\mathbb{A} = \{ z \in \mathcal{Z}: \|z\|> t_\alpha \}$ for some
(semi)-norm $\|\point\|$ on $\mathcal{Z}$ and a large threshold $t_\alpha$ chosen as
the $(1-\alpha)$-quantile of $\|Z\|$ where $\alpha=k/n $ throughout
this paper. Our main focus is on the case $\alpha \ll 1$.  In such a
context, it is natural to measure the performance of an algorithm in
terms of an expected loss, \emph{conditional} on the rare event
$\|Z\|>t_\alpha$. This generic setting encompasses in particular the
problem of classification in extreme regions 
introduced in \cite{jalalzai2018binary}.  The goal is to construct a
discrimination function which performs well on tail regions of the
covariates. In our context, this amounts to considering a pair
$z=(x,y)$ where $x\in\rset^d$ is the covariate and $y\in\{-1,1\}$ is
the target. The semi-norm on $\mathcal{Z}$ is then chosen of the form
$\|z\| = \|x\|_{\rset^d}$ where the latter denotes an arbitrary norm
on $\rset^d$. This setting has been applied to Natural Language
Processing in \cite{jalalzai2020heavy}, and finite-sample statistical
guarantees accounting from possible marginal standardization of the
covariates are analysed in~\cite{clemencon2022concentration}, Section
4.2. 
In this paper we take this specific problem as our leading example,
since classification by means of \ERM is a particularly
illustrative 
and simple 
statistical learning task. 
Further details regarding the classification setting and its relevance
for \EVA are postponed to Section~\ref{sec:framework} and
Section~\ref{sec:technicalClassif} in the supplementary material.
To illustrate the significance of our findings in terms of parameter
selection 
we consider a logistic-\Lasso regression algorithm trained on extremes
observations $\|Z\|\geq t_\alpha$, and we propose to choose the level
of the $\ell^1$-norm constraint on the parameter by 
a standard \Kfold CV
procedure. 
Such a sparsity-inducing classification algorithm is particularly
attractive with high dimensional covariates, as the curse of
dimensionality is particularly problematic in \EVA due to the reduced
effective sample size $k\ll n$.  Our results show that \Kfold permits
selecting a model within a finite collection in a risk consistent
manner.

In addition the classification example, 
our generic setting encompasses in principle also extreme quantile
regression. 
Indeed one may define a  semi-norm $\|(x,y)\| = \max(0,y)$ and $\mathbb{A}$ in terms of extreme values of the target $Y$,
typically $\mathbb{A} = \{(x,y): y \ge t(x) \}$  where $t(x)$ is a
high threshold which may depend of $x$. 
However the technical
additional assumptions that we make in Section~\ref{"ERM-notations"}
are taylored to our classification example, so that our results do not
apply immediately to extreme quantile regression. Further discussion
on this topic is deferred to Section~\ref{sec:framework},
Remark~\ref{rem:quantileReg}. 


\paragraph{\bf Related works regarding Cross-Validation.}


As mentioned above \CV\ may serve as a tool for 
$(i)$ risk estimation, $(ii)$ model selection. Sharp guarantees regarding the former task are typically needed as an intermediate step  to derive guarantees for the latter task, as discussed \emph{e.g.} in \cite{Laan2006}.  The model selection task  itself may be envisioned from the perspective of  \mbox{$(ii$--$a)$} estimation, where the goal is to minimize the risk attached to the final output, and
\mbox{$(ii$--$b)$} model identification,  where the goal is to select the `smallest' possible `true' model. In the present work we consider mainly task $(i)$. As a by-product, our results allow to derive minimal guarantees regarding task \mbox{$(ii$--$a)$}.
For an in depth review
of \CV\ for model selection and  risk estimation
 we refer the reader to
\cite{arlot2010survey} and the
reference therein or \cite{wager2020cross,bates2021cross}  for recent discussions.

A popular working assumption in the statistical learning literature  is \emph{algorithmic
  stability}
(\cite{Rogers1978,DEvroy-79,anthony1998cross,kearns1999algorithmic,bousquet2002stability}). In
this paper we adopt instead the framework of \ERM\ over a class of
predictors with finite VC-dimension in order to stay  close to  existing statistical learning viewpoints on \EVT\
mentioned above, leaving the question of stable algorithms to future
research.

Our main concern here is risk estimation  in  a non-asymptotic setting.
In this context, \cite{DEvroy-79}, \cite{anthony1998cross} and \cite{kearns1999algorithmic} show polynomial upper bounds on the \emph{leave-one-out} (\loo) error 
(\emph{i.e.} the probability of
deviation has a power decay) 
under various  weak stability assumptions, see also \cite{cornec10,cornec17}. In addition, \cite{kearns1999algorithmic} (Lemma 4.2) show that \ERM\ over a VC-class is in particular  error stable. Notice that exponential bounds 
(\emph{i.e.} ensuring that the probability of deviation is exponentially small) 
for the \loo\  
can be derived under the stronger assumption of \emph{uniform stability} as  \emph{e.g.} in \cite{bousquet2002stability}. In view of these facts it is reasonable to expect no more than a polynomial upper bound in our \ERM\ framework on extreme regions for the \loo and the \emph{leave-p-out} (\lpo) without further assumptions. Concerning the \Kfold\ scheme the literature is scarcer than for the \loo\ regarding upper bouds for risk estimation.  To our best knowledge the only existing non-asymptotic bounds in this respect are derived in \cite{cornec10,cornec17}. In the latter reference an upper bound is obtained for a wide range of \CV\ schemes. The bound incorporates a minimum between  an exponential and a polynomial terms, involving respectively the size of the validation and the training sets. This yields an exponential bound for the \Kfold\ since in the latter scheme the validation size is of the same order as the full sample size, contrarily to the \loo\ and \lpo\  
Both the exponential and polynomial upper bounds in the above references are \emph{sanity check} guarantees, which do not prove that the \CV\ risk estimate outperforms neither the hold-out nor the training risk estimate. However they prove that \CV\ is a consistent approach for risk estimation, and as a by-product, for model selection with an estimation purpose (task $(ii$--$a)$).  This is  not necessarily the case for model selection with an identification purpose  as discussed in several papers such as~\cite{wager2020cross}.

Going beyond sanity check bounds for \CV\ risk estimators remains an open challenge in the general case.
One natural question to ask 
is whether using
several training/testing folds improves upon the hold-out method (a
single split) or upon using the empirical risk on the training set
itself. 
Although the dependence between the different folds complicates the
analysis, partial answers have been brought in various specific
settings such as density estimation
\citep{arlot2008VFpen,arlot2016VFchoice} or \Lasso regression
\citep{homrighausen2013lasso,xu2020rademacher}. Restricting the analysis to the variance of the estimator, \cite{Blum1999} prove that  the \Kfold\  reduces the variance  and the amount of the latter is quantified in \cite{kale2011cross} and \cite{kumar2013near} under stability assumptions. Such improved guarantees in the context of rare events are left to future research.

Recent works in \EVT
\citep{boucheron2012concentration,boucheron2015tail,carpentier2015adaptive,goix15,jalalzai2018binary,lhaut2021uniform,clemencon2022concentration}
focus on finite-sample controls of the deviations of the empirical
measure in rare regions, with non-asymptotic upper bounds of the
desired order $1/\sqrt{k}$, thus matching the typical asymptotic rates
available from the Extreme Value literature (see \emph{e.g.}
\cite{de2006extreme}, chap. 3, 4). However the theoretical properties
of \CV estimates are notoriously difficult to establish due to the
lack of independence between the different terms of the average
involved in a \CV scheme. Our aim is to obtain \emph{sanity check
  bounds} \citep{kearns1999algorithmic,cornec10,cornec17} regarding
the deviations of the \CV estimate, that is, bounds that are of the
same order of magnitude as the ones regarding the empirical risk
itself, of order $\mathcal{O}(1/\sqrt{k})$ in our case, with
multiplicative constants depending on the complexity of the problem.

\paragraph{\bf Contributions.}
We provide three new results for \CV-based  risk estimation and model selection in a rare region of probability $\alpha\ll 1$:

\emph{(i)} An exponential probability bound involving the size of the validation set, which yields a sanity check bound in the context of rare events for the K-fold \CV scheme but not the \lpo\ scheme as the size of the validation set in this case remains constant, equal to $p$.

\emph{(ii)}  A polynomial upper bound, which outperforms the exponential one in the case of the \lpo because it only involves the size of the training set.

\emph{(iii)} For the sake of illustration, we apply our  exponential upper bound to the purpose of model selection within a finite family of models in logistic-\Lasso regression. In particular we obtain an upper bound on the excess risk  scaling as  $\mathcal{O}(1/\sqrt{n\alpha})$ \wrt the sample size with a multiplicative factor depending logarithmically on the number of candidate models.

 Both our contributions \emph{(i)} and \emph{(ii)} achieve state-of-the-art guarantees for $\alpha=1$, up to multiplicative constants and negligible terms. More precisely for $\alpha=1$  our exponential (\emph{resp.} polynomial) upper bound is of the same  nature as the ones in \cite{cornec17}  (\emph{resp.} \cite{kearns1999algorithmic} and \cite{cornec17}). However covering the case $\alpha\ll 1$ requires  different proof techniques accounting for the low variance (driven by $\alpha$) of the random variables at stake. In particular we use a Bernstein-type version of the bounded difference inequality due to \cite{McDiarmid98conc}, following in the footsteps of  previous statistical learning works devoted to \EVT\ mentioned above in the spirit of \cite{goix15} and \cite{lhaut2021uniform}. A distinctive challenge in the present work though is the complicated nature of the variable of interest, that is the cross-validation risk which involves a sum of dependent terms differently from the empirical risk studied in the latter references. 

  {\bf Outline.}
 The statistical framework envisioned in this paper is introduced in Section~\ref{sec:framework}. Our main results Theorems~\ref{theo:main-sanity-rare} and~\ref{theo:main-sanity-rare-lpo} 
 are presented respectively in Section~\ref{sec:sanity-exponential} and Section~\ref{sec:sanity-polynomial}. Guarantees regarding \Kfold for logistic-\Lasso regression are derived in Section~\ref{sec:applications}. 
 We illustrate the tightness of our bounds in numerical experiments reported in Section~\ref{sec:expes}. Possible extensions and leads for future research are discussed in Section~\ref{sec:disc}. The supplementary material includes additional details regarding Classification in Extreme Regions (Section~\ref{sec:technicalClassif}), generic statistical tools used in our proofs (Section~\ref{sec:genericTools}), as well as intermediate technical results and detailed proofs of our main results  (Section~\ref{sec:detailedProofs}).


\section{Extreme values, extreme risk and cross-validation: framework}\label{sec:framework}
\subsection{Conditional risk in an extreme region}

Consider a random element $Z$ valued in a sample space $\mathcal{Z}$
and a low probability region $\mathbb{A}\subset \mathcal{Z}$ such that
$\PP[Z\in \mathbb{A}] = \alpha$ with $0<\alpha \ll 1$. The
probabilistic behavior of $Z$ given that $Z\in \mathbb{A}$ is a main
concern in Extreme Value Analysis. Conditioning upon
$Z\in \mathbb{A}$, or alternatively rescaling the probability
distribution by an appropriate sequence, is a central idea in the
asymptotic \EVT literature related to the tail empirical processes,
see \emph{e.g.} the review paper from \cite{einmahl1992limit} or the
recent work from \cite{bobbia2021donsker} and the references therein
in relation to local empirical processes.  Uniform controls of the
deviations of the empirical measure based on an \iid sample
$Z_i, i\le n$ over such a region and, by extension, deviations of an
empirical risk conditional to $Z\in\mathbb{A}$, have been analyzed in
a non-asymptotic setting in several works over the past few years, in
various statistical contexts, such as empirical estimation of the
stable tail  dependence function \citep{goix15} and of the angular
measure \citep{clemencon2022concentration}, classification in extreme
regions \citep{jalalzai2018binary}, construction of Mass-Volume sets
for anomaly detection \citep{thomas2017anomaly}, support recovery
\citep{goix2017sparse}, principal component analysis
\citep{drees2021principal}, graphical models
\citep{engelke2021learning,engelke2022structure}. Recently
\cite{lhaut2021uniform} compute universal constants involved in the
upper bounds. They also discuss various conditioning and combinatoric
arguments and concentration tools for such non-asymptotic control.
 
Here we take as a leading example the problem of \ERM\ classification
in extreme regions, following in the footsteps of
\cite{jalalzai2018binary}, \cite{jalalzai2020heavy} and
\cite{clemencon2022concentration}. The sample space is
$\mathcal{Z} = \mathcal{X}\times\mathcal{Y}$ with
$\mathcal{X}\subset \rset^d$ and $\mathcal{Y}= \{-1,+1\}$.  The low
probability region of interest is then
$ \mathbb{A} = \{(x,y) \in \mathcal{Z}: \lVert x \rVert \geq t_\alpha
\}$ where $t_\alpha$ is the $1-\alpha$ quantile of a norm
$\lVert X \rVert$ on $\mathcal{X}$.  Notice that in our setting, the
probability $\alpha$ is known (chosen by the user) whereas the
threshold $t_\alpha$ -- thus also $\mathbb{A}$ -- is unknown because
the law of $X$ is unknown. Given a class $\mathcal G$ of
discrimination rules $g: \mathcal{X}\to \rset$ and a loss function
$c : \mathcal G \times \mathcal Z \to
\rset$, 
the conditional risk of $g\in \mathcal G$ over the rare region
$\mathbb{A}$ is 
\begin{equation}
\label{eq:def-risk-extreme}
\mathcal{R}_{\alpha}(g)=\EE\left[ c(g,Z)\mid Z \in \mathbb{A} \right].
\end{equation}

Upon appropriate regularity assumptions (regular variation of class
distributions), it is possible to properly define a nontrivial
classification problem conditional to $\|x\|>t_\alpha$, in the limit
$t_\alpha\to\infty$.  Namely, 
for the $0-1$ loss associated with binary classifiers, under
appropriate regular variation assumptions regarding the class
distributions $\mathcal{L}(X | Y= \sigma 1), \sigma\in\{+,-\}$, the
conditional risk $\mathcal{R}_\alpha$ of an \emph{angular} classifier
of the kind $g(x) = \tilde g(\theta(x))$, with
$\theta(x) = \|x\|^{-1}x$, converges as $\alpha\to 0$ to an asymptotic
risk $\mathcal{R}_\infty$. The latter is the out-of-sample risk of $g$
in the extreme region.
Importantly, it is shown in the latter reference that learning an
angular classifier with nearly optimal properties above an
asymptotically high level can be done based on the few largest
training points (where `large' should be understood in the sense of
the norm of the covariates). 
For  
self-containment we provide a brief account of this
framework in Section~\ref{sec:backgroundClassif} of the supplementary
material. 

%

\begin{example}[Prediction in heavy-tailed vectors]\label{ex:relevanceRalpha-RV}
  We provide here an example of application of the classification setting considered in this paper to a traditional \EVA task, which is predicting the occurrence of an extreme event. Indeed predicting a binary output $Y$ based on large covariates may seem somewhat disconnected from standard \EVA frameworks typically concerned with the analysis of extreme risks. The  present example has not been considered in the original paper \cite{jalalzai2018binary} and highlights the relevance of our working assumptions. 
  Consider the general problem of predicting the  missing value of one component (say $Z_{d+1}$)  of a random vector $Z = ( Z_1,\ldots, Z_{d+1})\in\rset^{d+1}$,  based on the partial observation  $(Z_{1}, \ldots, Z_d)$,  given that the latter  is large. 
  A first step would be to predict whether $Z_{d+1}$ would also be large. With this in mind let
\begin{equation*}
  X  = (Z_1,\ldots, Z_{d}) \text{ and }
  Y = \un{ \frac{Z_{d+1}}{\| (Z_1,\ldots, Z_{d+1}) \| }> c  } ,  \end{equation*}
  where $\|\point\|$ is the $\ell^p$ norm for some $p\in[1,\infty)$ and  $c \in (0,1)$ may  be chosen depending on the downstream task. As an example   $c=(1/(d+1))^{1/p}$ if the  target event is that $Z_{d+1}>0$ and  $|Z_{d+1}|^p$ is at least as large as the average value of the $|Z_j|^p$ 's for $j\le d+1$.  
  We prove in Appendix~\ref{sec:appliPrediction} that the pair $(X,Y)$ thus defined satisfies the requirements of \cite{jalalzai2018binary}'s setting, if $Z$ is  a heavy-tailed random vector $Z = (Z_1,\ldots, Z_{d+1})\in\rset^{d+1}$, more precisely, 
  if $Z$ has a density with respect to the Lebesgue measure which is regularly varying. This   `regular variation of densities' assumption  has been used on multiple occasions in \EVT (\cite{de1987regular,Cai2011}).  
  
  In addition, the recent work~\cite{huet2023regression} extends the binary classification framework on extreme covariates, to the case where the goal is to predict a continuous target $Y\in[-M,M]$. Our Assumptions 1--4   encompass  also this regression framework since the squared error loss function is bounded as long as  the prediction functions $g$ are chosen in a uniformly bounded class and $Y$ is bounded.  Again, the requirement that $Y$ is bounded may seem at odds with standard \EVA applications. However the authors show (Example~2 in the cited reference) that this regression setting encompasses prediction of (missing)  components in multivariate regularly varying vectors, given that the other components are large. This is achieved by setting, with the same notations as above, $Y = Z_{d+1}/\|(Z_1, \ldots, Z_{d+1})\|$. Then any prediction $\hat Y$ for $Y$ may be plugged-in to obtain a prediction $\hat{Z}_{d+1}$ using the identity
  $ Z_{d+1} = Y \| (Z_{1}, \ldots, Z_d)\|_p /( 1 - |Y|^p )^{1/p} $.  
For simplicity, we focus solely on the classification task  in our application to LASSO (Section~\ref{sec:exampleLogisticRegression} and Section~\ref{sec:applications}) and simulation study (Section~\ref{sec:expes}).   
\end{example}

  \begin{remark}[Another example: Extreme quantile regression]\label{rem:quantileReg}
   Another natural example of a supervised problem mentioned in the introduction  is extreme quantile regression, which is a well-documented topic in EVT (see \emph{e.g.} the book chapter~\cite{chernozhukov2017extremal} and the references therein, \cite{daouia2013kernel}) with some extensions based on expectiles~(\cite{girard2021extreme}). In this case, the
  target variable $Y$ is generally unbounded and takes values in
  $\rset_+$.  Extreme events are then defined relatively to $Y$.  Our general problem formulation also encompasses this situation 
 upon defining the semi-norm on $\mathcal{Z}$ as $\|z\| = y$.
  In~\cite{farkas2021cyber},~\cite{velthoen2021gradient} and~\cite{gnecco2022extremal}, the authors place themselves in a parametric setting and consider respectively regression trees, tree-based gradient boosting, and extremal  random forest, for regressing  GPD parameters on the covariates and obtain extreme quantiles by a plug-in strategy.
As the three latter methods are likelihood-based,  cross-validating a likelihood criterion comes as a natural idea for  goodness-of-fit assessment, or for the choice of  hyper-parameters. In practice  the number of splits in~\cite{farkas2021cyber} is chosen using a complexity penalty, with a multiplicative factor which is indeed chosen by \CV. Similarly~\cite{velthoen2021gradient} and~\cite{gnecco2022extremal} respectively recommend choosing the number of trees and the minimum node size by \Kfold \CV. 

Mathematically speaking, classification  is
  somewhat simpler than quantile regression,  because the
  loss functions involved in ERM can be chosen bounded, which matches
  our Assumption~\ref{assum:cost-func} (Section~\ref{"ERM-notations"}).  Relaxing our boundedness assumption would extend the scope of application of our results in the next section  to extreme quantile regression. We conjecture that it is possible to do so with an adequate control of the tails of the loss $c(g,Z)$. This is left for future works. 
  \end{remark}

In practice, the quantity of interest is not the risk of a fixed 
discrimination function $g$, but the risk of the specific $\hat g$ issued by an algorithm, also called \emph{learning rule}, given training data $\DD_n=(Z_1,\ldots, Z_n)$. Here and throughout, we assume that $\DD_n$ is  a collection of independent and identically distributed random vectors with common distribution $P$ and $S_n=\{ 1, \ldots, n \}$ refers to the full index set. 
Formally,  a  learning rule can be viewed as a function  $\Psi: \sqcup_{m \le n} \mathcal{S}_m \to \mathcal{G}$, where $\mathcal{S}_m$  is the family of subsets of $\{1,\ldots,n\}$ of size $m\le n$, and $\Psi(S) = \hat g$ is the output of the rule trained on $\{Z_i,i\in S\}$.
The quantity that we would like to estimate is the generalization risk of the learning rule trained on $\DD_n$, $\mathcal{R}_{\alpha}\big(\Psi(S_n)\big)$. 

Given a subsample $S\subset \{1,\ldots , n\}$, for some fixed $g$, an  empirical version of $\mathcal{R}_\alpha(g)$   based on  $S $ is 
\begin{equation}
  \label{eq:risk_est}
\widehat{\mathcal{R}}_{\alpha}(g, S)=\frac{1}{\alpha n_S} \sum_{i\in S} c(g,Z_i)\indic{\lVert X_{i}\rVert>\lVert X_{(\lfloor \alpha n \rfloor)}\rVert },  
\end{equation}
where $n_S=card(S)$ and $\left\|X_{(1)}\right\| \geq \ldots \geq\left\|X_{(n)}\right\|$ are the (reverse) order statistics of the sample $(\left\|X_i\right\|)_{i= 1,\ldots, n }$.

 Notice that the random threshold $\|X_{(\lfloor \alpha n \rfloor)}\|$ used for selecting extreme observations in the risk estimate (\ref{eq:risk_est}) is defined using the full  index set  $S_n$, not the particular  subsample $S$. An alternative strategy would be to let the random threshold depend on the particular subsample $S$, using \emph{e.g.} the $\lfloor \alpha n_S \rfloor^{th}$ order statistic within $S$. In the present work we limit ourselves to the analysis of \CV estimates of the risk based on the common threshold $\|X_{(\lfloor \alpha n \rfloor)}\|$ which turns out to be  convenient in our proofs, see \emph{e.g.} the argument leading to~(\ref{'c_ext'}) in the Appendix. 
 Whether it is possible to obtain similar or better guarantees for the alternative strategy based on a variable threshold remains an open question and would require in any case a substantial modification of our proof techniques. 

Equipped with the  definition of an empirical risk $\widehat{\mathcal{R}}_{\alpha}(g,S)$ in~\eqref{eq:risk_est},   the \emph{hold-out} estimator of the  risk $\mathcal{R}_\alpha(\Psi(S_n))$ based on a  validation set $V\subset \{1,\ldots, n\}$ and a training set $T = \{1,\ldots, n\}\backslash V$ takes the simple form
$	\widehat{\mathcal{R}}_{\alpha}( \Psi({T}) , V)$. The  \CV\ strategy for estimating $\mathcal{R}_\alpha(\Psi(S_n))$ consists in averaging  such hold-out estimates over a family of 
validation sets  $V_{1:K} = (V_j)_{j=1,\ldots, K}$,  where $V_j\subset\{1,\ldots,n\}$. 
Namely the \CV\  estimator 
is

\begin{equation}\label{"rcvEx-def"}
	\widehat{\mathcal{R}}_{CV,\alpha}(\Psi,V_{1 : K})=
	\frac{1}{ K} \sum_{j= 1}^{K} 
	\ \widehat{\mathcal{R}}_{\alpha}( \Psi({T_j}) , V_j),
\end{equation}
where $T_j = \{1,\ldots, n\} \backslash V_j$.

\begin{remark}[Focus on the estimation error at fixed level $\alpha$  and bias term]\label{rem:EstimationError-noRV}  
  The \ERM strategy proposed in \cite{jalalzai2018binary} and \cite{clemencon2022concentration} in order to choose an appropriate classifier $\hat g$ regarding $\mathcal{R}_\infty$ consists in minimizing the  empirical version of the subasymptotic risk $\mathcal{R}_\alpha$, 
  $\widehat{\mathcal{R}}_\alpha(g, S_n)$, where $S_n$ is the full index set $\{1,\ldots, n\}$. The statistical guarantees obtained in these papers concern the uniform deviations of $\widehat{\mathcal{R}}_\alpha$ and as a consequence the excess  $\mathcal{R}_\alpha$-risk of the empirical risk minimizer $\hat g$. The bias term $\mathcal{R}_\infty - \mathcal{R}_\alpha$ is left aside from  their statistical analysis. However it is shown in \cite{clemencon2022concentration} (Remark 3.3 and Appendix D) that for typical multivariate heavy-tailed vectors such as  multivariate Cauchy random variables, the bias is of order $\mathcal{O}(\alpha)$ and is thus negligible compared with deviation terms.  

  In the present work, we take a similar approach  in that  our main focus is on  \CV-based estimation of $\mathcal{R}_\alpha$. We thus leave the bias term outside our scope. One benefit from this strategy is that our work does not rely on any regular variation assumptions, leaving open the possibility of applications to other contexts outside  \EVA where rare events play a major role, such as anomaly detection or imbalanced classification.  
\end{remark}

\subsection{Working assumptions}\label{"ERM-notations"}

Our main results hold under Assumptions~\ref{assum:ERM-settings} to~\ref{assum:cost-func}  introduced below.

\begin{assumption}[ERM algorithm] \label{assum:ERM-settings}
  The learning rule denoted by $\Psi_\alpha$,   is  an   empirical conditional risk minimizer for the probability level $\alpha$, 
	\begin{equation}\label{"g_alpha_j_def"}
	\Psi_\alpha({S})=\argmin_{g \in \mathcal{G}}\widehat{\mathcal{R}}_{\alpha}(g, S).
	\end{equation}	
\end{assumption}

For clarity reasons, we suppose further that $n$ is divisible by $K$ so that $n/K$ is an integer. This condition guarantees, in the case of $K$-fold cross-validation, that all validation sets have the same cardinal $n_{V}=n/K$. We also need the sequence of validation/training sets to satisfy a certain balance condition which is expressed below.

\begin{assumption}[\CV scheme balance condition]\label{assum:mask-property}
	The sequence of validation sets  $V_1,V_2,\dots V_K$  satisfies
	\begin{equation}\label{"card-condition"}
	card(V_j)=n_V\quad \forall j \in \{1,\ldots, K\},
	\end{equation}
	for some $n_V \in \{1,\ldots, n\}$.\,Moreover it holds that 
	\begin{equation}\label{"key-condition"}
	\frac{1}{K}\sum_{j=1}^{K} {\indic{l \in V_j} }=\frac{n_V}{n} \quad \forall l \in \{1,\ldots, n\}. 
	\end{equation}
\end{assumption}

The next lemma ensures that Assumption~\ref{assum:mask-property} holds true for the standard \CV  procedures (K-fold and \lpo) and that an identity similar to~\eqref{"key-condition"} is also valid for the training sets $T_j$. The proof is provided in Appendix~\ref{"proof-train-test-cond"}.
\begin{lemma}
	\label{"train-test-cond"}
	If $K$ divides $n$, for the leave-one-out, the
        leave-$p$-out,\,and the $K$-fold procedures, the validation
        sets $V_{1:K}$ satisfy
        Assumption~\ref{assum:mask-property}. Also the training
        sets 
	$T_{1:K}$
        satisfy 
	\begin{equation*}
	\frac{1}{K}\sum_{j=1}^{K}\ffrac{\indic{l \in T_j} }{n_{T}} 
        =\frac{1}{n} \quad \forall l \in \{1,\ldots, n\}.
	\end{equation*}
	
\end{lemma}

\begin{remark}\label{rem:KdividesN}
  The condition that $K$ divides $n$ is required for the  $K$-fold  \CV only, to ensure that  $card(V_j)=n_{V}$ for all $j$. 
  However straightforward extensions of our results can be obtained in the case where $K$ does not divide $n$ at the price of some notational complexity. 
\end{remark}

We now introduce two assumptions relative to the function class $\mathcal G$ and the cost function $c$. They shall be useful to control the fluctuation of the underlying empirical process. First we require the following standard complexity restriction on the family of functions $(x,y) \mapsto c(g(x), y)$ when $g$ lies in $ \mathcal G$. 
  We refer to \cite{Vapnik1998} for standard definitions of Vapnik-Chervonenkis (VC) dimensions and classes,  and standard concepts of statistical learning.

\begin{assumption}[Finite VC-dimension]\label{assum:hypo-class}
  The family $\mathcal{G}$ of predictors and the cost function $c$ are such that the class of functions $\{z\mapsto c(g,z) = c(g(x),y), g\in\mathcal{G}\}$ 
  has a finite Vapnik-Chervonenkis (VC)-dimension $\vapnikG$, \ie the family of subgraphs
  $\big\{ \{(x,y, t): t < c(g(x), y) \}: t \in \rset, (x,y) \in \mathcal{Z}, g \in \mathcal{G} \big\}$ has VC-dimension $\vapnikG$.
      \end{assumption}
This complexity assumption mainly  allows us to use a uniform concentration inequality from \cite{gine2001consistency}, which requires that the covering number for the $L_2$ norm of this family of functions decrease polynomially (with exponent $\vapnikG$) (see their condition (2.1)). We may thus as well assume the latter weaker condition, which is sometimes easier to check in practice,  instead of Assumption~\ref{assum:hypo-class},  without altering our results.   
                  
For simplicity  we limit ourselves to a cost function bounded by $1$. Our result may be extended to any bounded cost function at the price of a multiplicative scaling factor.
\begin{assumption}[Normalized cost function]\label{assum:cost-func}
	The cost function $c$ is non-negative and bounded by $1$,
	\[ 0 \leq c(g,Z) \leq 1 \quad\quad \forall (g,Z) \in  \mathcal{G}\times \mathcal{Z}. \]
     This hypothesis is clearly satisfied for the  Hamming loss $c(g,Z) = \un{ g(X) \neq Y }$.
   \end{assumption}
   \begin{remark}[Boundedness assumption]\label{rem:bounded}   
The logistic loss considered in our illustrative application (Section~\ref{sec:exampleLogisticRegression} and Section~\ref{sec:applications}) is not bounded in general. However, in the context of classification in extreme regions,  we consider angular classifiers, with a $\ell_1$ constraint on the parameter $\beta$, which amounts to a boundedness assumption on the loss. Indeed
$|\beta^\top\theta| \le \max_{j\le d} |\theta_j|\sum_{j\le d} |\beta_j|$, whence, for any $t>0$, 
$\sup_{\|\beta\|_1\le t, \|\theta\|_\infty = 1 } |\beta^\top\theta| \le t$ for any $t>0$.
  Of course, it would be interesting to weaken this boundedness assumption regarding the cost to extend the scope of our theoretical results. 
\end{remark}

\subsection{Illustration: selecting a regularization parameter in high-dimensional classification on extreme covariates
}\label{sec:exampleLogisticRegression}

As an illustration of how \CV may be used for model selection in the classification setting introduced in Section~\ref{ex:relevanceRalpha-RV} under our working assumptions,  we 
consider the typical problem of selecting a regularization parameter
for 
high-dimensional classification. When the dimension $d$ of the
covariate variable $X$ is large, a well documented way to reduce the
dimension of the predictor is to add a non-differentiable penalty term
to a (convex) \ERM problem, or equivalently to solve a convex
minimization problem under sparsity inducing
constraints. 
We thus consider a LASSO-type logistic regression problem with
discrimination functions $g_\beta$ indexed by a $d$-dimensional
parameter $\beta$. In Section~\ref{sec:surrogateLoss} of the
supplementary material we prove that the main findings of
\cite{jalalzai2018binary} remain valid with the logistic loss instead of the Hamming loss  originally considered in~\cite{jalalzai2018binary}.
These findings suggest restricting the
attention to angular discrimination functions.  In this context,
$g_\beta(x)= \beta^\top\theta(x)$. Recall that
$\theta(x) = \|x\|^{-1}x$ for some norm $\|\,\cdot\,\|$. In our
experiments we shall choose the sup-norm.  The logistic loss function
is then, for $z = (x,y)\in \rset^d\times\{-1,1\}$, 
        \begin{equation*}
          c(g_\beta, z)   =  \log \left( 1+\exp\left(- \beta^\top \theta(x)  x \right)\right). 
\end{equation*}
 Because it is mathematically convenient to handle bounded losses, we  impose (see Assumption~\ref{rem:bounded}) that the  cost functions are bounded. This is not the case in general for the logistic loss above without further assumptions on $\beta$. For this reason we restrict the scope of our analysis to the constrained logistic LASSO, \ie we impose a hard $\ell_1$ constraint on the parameter $\beta$. The learning rule $\Psi_{\alpha,u} $ for the constrained Logistic-\Lasso  with $\ell_1$ constraint $\|\beta\|_1\le u$ is thus
\begin{equation}\label{def:lasso-constraine}
  \begin{aligned}
          \widehat \beta_u  &=  \Psi_{\alpha,u}(S) = 	\argmin_{g \in {\mathcal{G}}_u } \ER_\alpha(g, S),   \\
     \text{ where }      \mathcal{G}_u & = \{g_\beta, \beta\in \rset^d, \|\beta\|_1\le u  \}\\
     \text{ and } \ER_\alpha(g_\beta, S) &  = \frac{1}{\alpha n_S} \sum_{i\in S}  \log \left( 1+\exp\left(- \beta^\top \theta(X_i)  Y_i \right)\right)\un{ \| X_i\|  > \| X_{(\lfloor \alpha n \rfloor )}\|}. 
  \end{aligned}
\end{equation}

\begin{remark}[Binary versus continuous outputs, surrogate loss functions]\label{rem:convexLoss}
In \cite{jalalzai2018binary} and \cite{clemencon2022concentration} the  analysis is limited to binary classifiers $g(x) \in \{-1,1\}$ and  the $0$-$1$ loss $c(g,(x,y))= \un{g(x)\neq y}$. Extending their results to more general cost functions such as convex surrogate losses would be an interesting avenue for further research, leveraging  ideas summarized in the review paper from \cite{boucheron2005theory}, Section~4, and the references therein, in particular \cite{zhang2004statistical}. This would allow  to cover the case of computationally realistic algorithms such as Support Vector Machines or logistic regression.  In the present paper we  take a step towards this end and consider general real-valued discrimination functions $g$ with a bounded loss function $c$, as made precise in our working assumptions~\ref{assum:ERM-settings}--\ref{assum:cost-func} listed in Section~\ref{"ERM-notations"} and exemplified with the logistic example. 
We analyse   the deviations of the \CV\ estimate with respect to the  (constrained) logistic expected loss itself. However we do not relate the latter convex surrogate loss  with the deviations of the $0$-$1$ error, 
a task which could be the subject of further work.        
\end{remark}

\begin{remark}[Extensions]\label{rem:extensionsOtherCosts}
The logistic loss and the $\ell_1$ constraint (or penalty) are  one of many pairs (convex surrogate loss function - penalty)  commonly considered in statistical learning. As an example, soft margin Support Vector Machines rely on the pair (hinge loss+ $\ell_2$ norm). 
We only consider in Section~\ref{sec:applications} the particular example of logistic regression under $\ell_1$ constraint, for the sake of concreteness and simplicity. 
\end{remark}


\section{Exponential bounds for K-fold \CV  estimates in rare regions}\label{sec:sanity-exponential}

Our first main result Theorem~\ref{theo:main-sanity-rare} below holds
true for any \CV procedure under Assumptions~\ref{assum:ERM-settings}~--~\ref{assum:cost-func}. The leading term of the provided upper bound
is $\mathcal{O}(\sqrt{\vapnikG/(n_{V}\alpha)})$. In the case of the
$K$-fold $1/n_{V} = \mathcal{O}(1/n)$. Thus, the
bound of Theorem~\ref{theo:main-sanity-rare} becomes
$\mathcal{O}(\sqrt{\vapnikG \log(1/\delta)/(n\alpha)})$. The
latter bound is indeed a sanity check bound as it matches (up to
unknown multiplicative constants) the one relative to the empirical
risk conditional to a rare event established in
\cite{jalalzai2018binary},~Theorem~2, where $k=n\alpha$.

\begin{theorem}[Exponential \CV bound for rare events]\label{theo:main-sanity-rare}
  Under assumptions~\ref{assum:ERM-settings}, \ref{assum:mask-property}, \ref{assum:hypo-class}, \ref{assum:cost-func}, we have, with probability $1-15\delta$,
 \begin{align*}
   \bigg|\widehat{\mathcal{R}}_{CV,\alpha}(\Psi_{\alpha},V_{1:K})-
   \mathcal{R}_{\alpha}\big(\Psi_{\alpha}(S_n)\big)\bigg|  \leq  E_{CV}(n_{T},&n_{V},\alpha) 
   + \frac{20}{3 n \alpha} \log\left(\frac{1}{\delta}\right)\\&   +  20\sqrt{\frac{2}{n\alpha} \log\left(\frac{1}{\delta}\right)},
\end{align*}
where 
\begin{align*}
E_{CV}(n_{T},n_{V},\alpha)=M\sqrt{\mathcal{V}_\mathcal{G}}\left(\frac{1}{\sqrt{n_{V}\alpha}}+\frac{4}{\sqrt{n_{T}\alpha}}\right)+\ffrac{5}{n_{T}\alpha},
\end{align*}  
and where $M >0$ is a universal constant.
\end{theorem}

\begin{proof}[Sketch of the proof]

Introduce the pseudo-empirical risk 
\begin{equation}\label{def:tilde-rcv}
	\Tilde{\mathcal{R}}_{\alpha}(g , S)=\frac{1}{\alpha n_S} \sum_{i\in S} c(g, Z_i)\indic{\lVert X_{i}\rVert>t_\alpha }.
\end{equation}
Notice that when the distribution of $\lVert X \rVert$ is unknown,  $\Tilde{\mathcal{R}}_{\alpha}$ is not observable and only $\widehat{\mathcal{R}}_{\alpha}$ is a genuine statistic.\,However $ \Tilde{\mathcal{R}}_{\alpha}$ will serve as an intermediate quantity in the proofs. \\
Define the average  pseudo-empirical risk of the family $\big(\Psi_\alpha(T_j)\big)_{  1\leq j \leq K}$ by
\begin{equation}\label{eq:pseudo-rcv-def}
\Tilde{ \mathcal{R}}_{CV,\alpha}(\Psi_{\alpha},V_{1:K})=\frac{1}{K}\sum_{j=1}^{K}\Tilde{\mathcal{R}}_{\alpha}(\Psi_{\alpha}(T_j),V_j)
\end{equation}
and the average `true' risk  by 
\begin{equation}\label{eq:true-rcv-def}
	\mathcal{R}_{CV,\alpha}(\Psi_{\alpha},V_{1:K})=\frac{1}{K}\sum_{j=1}^{K}\mathcal{R}_{\alpha}( \Psi_{\alpha}(T_j)).
\end{equation}

Using the previous quantities,\,write the following decomposition
\begin{align}\label{ineq:main-decomposition}
  \bigg|\widehat{ \mathcal{R}}_{CV,\alpha}(\Psi_{\alpha},V_{1:K})& -\mathcal{R}_{\alpha}\big(\Psi_{\alpha}(S_n)\big)\bigg| 
  \leq 
	\DevExt+\DevCV+\BiasCV, 
\end{align}
with 
\begin{align}
	&\DevExt=\abs{\widehat{ \mathcal{R}}_{CV,\alpha}(\Psi_{\alpha},V_{1:K})- \Tilde{ \mathcal{R}}_{CV,\alpha}(\Psi_{\alpha},V_{1:K})}\label{"dev_ext"},\\
	&\DevCV= \abs{\Tilde{ \mathcal{R}}_{CV,\alpha}(\Psi_{\alpha},V_{1:K})-  \mathcal{R}_{CV,\alpha}(\Psi_{\alpha},V_{1:K})}\label{"dev_CV"}, \\
	&\BiasCV= \abs{\mathcal{R}_{CV,\alpha}(\Psi_{\alpha},V_{1:K})-\mathcal{R}_{\alpha}\big(\Psi_{\alpha}(S_n)\big)}. \label{"bias_CV"}
\end{align}
The remainder of the proof (see Section~\ref{sec:main-theo-proof}) consists in deriving upper bounds for each terms of the error decomposition~(\ref{ineq:main-decomposition}),  from which  the result follows.

The term $\DevExt$
measures the deviation between the cross-validation  estimator when
using the order statistics and the cross-validation  estimator when
using the `true' level $t_\alpha$, which can be bounded using Bernstein inequality, taking advantage of the small variance of the random indicator function $\un{\|X\| > t_\alpha}$.
The term $\DevCV$ measures the
deviations of
$\Tilde{ \mathcal{R}}_{CV,\alpha}(\Psi_{\alpha},V_{1:K})$ from its
mean. It is controlled by a uniform bound (over the class $\mathcal{G}$) on the deviations of the empirical risk evaluated on the validation sample. To do so we leverage recent arguments leading to a bound on such deviations on low probability regions (as \emph{e.g.} in \cite{goix15} and \cite{jalalzai2018binary}). Finally the term 
$\BiasCV$ is the bias of the cross-validation 
procedure, the control of  which relies on the specific nature (ERM) of the considered learning algorithm. Indeed in this context the bias may be upper bounded in terms of the supremum deviations of the empirical risk evaluated on the training sets $T_j$.  Notice that the  bias term  discussed  in Remark~\ref{rem:EstimationError-noRV}   resulting from the threshold choice, and the bias term in~\eqref{"bias_CV"} resulting from the choice of the number of folds (or equivalently, from the size of the training sets $T_j$ in the validation procedure) are of different nature. The former relates to the  discrepancy between the distribution of available extremes and the limit distribution. The latter relates to the discrepancy between the error of an algorithm trained with $card(T_j)$ observations, and the error of an algorithm trained with $n$ observations.
\end{proof}

Theorem \ref{theo:main-sanity-rare} can be used  to obtain exponential bounds for the $K$-fold \CV estimate. From Lemma~\ref{"train-test-cond"}, 
Assumption~\ref{assum:mask-property} regarding the sequence of validation sets $V_{1:K}$ holds true for the K-fold \CV procedure. Consequently Theorem~\ref{theo:main-sanity-rare} applies  with $n_{V}=n/K$ and $n_{T}=n-n_{V}=\frac{K-1}{K}n$. 
In the following corollary, $V_{1:K}^{\Kfold}$ denotes the sequence of validation sets associated with $K$-fold.

\begin{corollary}\label{coro:K-fold-sanity-check}
	Under the assumptions of Theorem~\ref{theo:main-sanity-rare}, the $K$-fold  \CV estimate (with $K\geq 2$) for the conditional risk (\ref{eq:def-risk-extreme}) satisfies with probability $1-15\delta$,
	\begin{align*}
          \bigg|\widehat{\mathcal{R}}_{CV,\alpha}(\Psi_{\alpha},V_{1:K}^\Kfold)-
          \mathcal{R}_{\alpha}\big(\Psi_{\alpha}(S_n)\big)\bigg|  \leq  E_{\Kfold}(n,&K,\alpha)
          + \frac{20}{3 n \alpha} \log\left(\frac{1}{\delta}\right)\\
          &   +  20\sqrt{\frac{2}{n\alpha} \log\left(\frac{1}{\delta}\right)},
	\end{align*}
	with 
	\begin{align*}
	E_{\Kfold}(n,K,\alpha)=5M\sqrt{\frac{\mathcal{V}_\mathcal{G}K}{n\alpha}}  +\ffrac{5K}{(K-1)n\alpha}.
	\end{align*}  
\end{corollary}

\begin{discussion*}
As mentioned in the introduction, for $\alpha=1$ the upper bound in Theorem~\ref{theo:main-sanity-rare} and its application to the \Kfold\ in Corollary~\ref{coro:K-fold-sanity-check} are of the same nature as in \cite{cornec17}, Proposition 4.1,  which apply to the same context as ours, \ie\ \ERM\ over a hypothesis class of finite VC-dimension. 
Covering the case $\alpha\ll 1$ requires special proof techniques with a Bernstein-type inequality due to \cite{McDiarmid98conc} and recalled in Proposition~\ref{"Bernstein-extension"} in the supplement. Doing so  improves by a factor $\sqrt{\alpha}$ over the naive method consisting in diving both sides of the existing bounds by $\alpha$. As discussed in the introduction this naive method  yields a potentially diverging bound as $\alpha = \alpha(n)=\frac{k}{n}\to 0$. Also the organization of our proof is different, in particular a key simplifying step  is the balance condition of the \CV\ schemes which applies to the \Kfold\ (Lemma~\ref{"train-test-cond"}), a fact which (to our best knowledge) is not mentioned in the existing literature. Finally, though \cite{kumar2013near} quantify the amount of variance reduction brought by the \Kfold, the bias term is left outside the analysis in this reference.  We haven't found any comparable finite sample upper bound in the literature devoted to stable algorithms, a natural question to ask since \ERM over a VC class is error stable \citep{kearns1999algorithmic}.
\end{discussion*} 
\begin{remark}[On the universal constants]
  \label{rem:unknownConstants}
  A drawback of our results in the present section and in the
  following one (Section~\ref{sec:sanity-polynomial}) is the presence
  of universal constants in our upper bounds. These
  unknown constants derive from our control of the Rademacher averages
  using \cite{gine2001consistency}, who themselves resort to chaining
  arguments. This is a standard issue in statistical learning. In most
  cases these constants may be replaced with additional logarithmic
  terms with respect to the sample size or may be computed
  explicitly. For the empirical training risk in low probability
  regions these improvements are respectively achieved
  in~\cite{lhaut2021uniform} and in~\cite{clemencon2022concentration},
  Theorem A.1. We leave this question for further research regarding the \CV\ risk. 
\end{remark}

Despite the satisfactory
sanity check bound obtained thus far for the $K$-fold (Corollary~\ref{coro:K-fold-sanity-check}), note that the term
$\mathcal{O}\left(\sqrt{\mathcal{V}_\mathcal{G}/(n_{V}\alpha)}\right)$ in the upper bound of Theorem~\ref{theo:main-sanity-rare}  does not even
converge to $0$ in the \lpo setting because the size $n_V$ of the
validation set remains constant, equal to $p$. Thus, Theorem~\ref{theo:main-sanity-rare} is not adapted to the latter type of \CV
schemes. In the next section we obtain
(Theorem~\ref{theo:main-sanity-rare-lpo}) an alternative upper-bound
involving only the size $n_T$ of the training set which allows to cover the \lpo case.


\section{Polynomial bounds for \lpo \CV estimates in rare regions}
\label{sec:sanity-polynomial}
Theorem~\ref{theo:main-sanity-rare} provides trivial bounds for \CV schemes with small test size. In contrast our second main result (Theorem~\ref{theo:main-sanity-rare-lpo} below) 
yields a sanity-check bound for a wider class of \CV
procedures, including leave-one-out and leave-p-out.  In particular, we show
that  with high probability, the error is at most
$\mathcal{O}(\sqrt{\mathcal{V}_\mathcal{G} /(n_{T}\alpha)})$. Most, if not all  \CV procedures satisfy
$1/n_{T}=\mathcal{O}(1/n)$ and the latter bound is
thus of order
$\mathcal{O}(\sqrt{\mathcal{V}_\mathcal{G} /(n\alpha)})$.
\begin{theorem}[Polynomial cross-validation bounds for rare events]\label{theo:main-sanity-rare-lpo}
	Under assumptions~\ref{assum:ERM-settings}, \ref{assum:mask-property}, \ref{assum:hypo-class},\,\ref{assum:cost-func} one has with probability $1-17\delta$,
	\begin{align*}
		\big|\widehat{ \mathcal{R}}_{CV,\alpha}(\Psi_{\alpha},V_{1:K})&-\mathcal{R}_{\alpha}\big(\Psi_{\alpha}(S_n)\big) \big| \leq 
		E'_{CV}(n_{T},\alpha)+
		\ffrac{1}{\delta\sqrt{n_{T}\alpha}}(5M\sqrt{\mathcal{V}_{\mathcal{G}}}+M_5),
	\end{align*}
	
	where $M,M_5>0$ are universal constants, $M$ is the same as in Theorem~\ref{theo:main-sanity-rare} 
         and
         \[E'_{CV}(n_{T},\alpha)= \ffrac{9M\sqrt{\mathcal{V}_\mathcal{G}}}{\sqrt{\alpha n_{T}}}+\ffrac{9}{n_{T}\alpha}.
         \]
\end{theorem}
\begin{proof}[Sketch of the proof]
First write 
\begin{align}
\left|\widehat{ \mathcal{R}}_{CV,\alpha}\right.&\left.(\Psi_{\alpha},V_{1:K})-\mathcal{R}_{\alpha}\big(\Psi_{\alpha}(S_n)\big)\right|  \leq \BiasCV+  \abs{\widehat{ \mathcal{R}}_{CV,\alpha}(\Psi_{\alpha},V_{1:K})-\mathcal{R}_{CV,\alpha}(\Psi_{\alpha},V_{1:K})} \label{"basic-inequality"}\,,
\end{align}
where $\BiasCV$ is defined by \eqref{"bias_CV"}.

The upper bound for the term $\BiasCV$ obtained in the proof of
Theorem~\ref{theo:main-sanity-rare} is of order
$\mathcal{O}(1/\sqrt{n_{T}\alpha})$, see~\eqref{"C_train_bound"} in the supplement for details. Since $1/n_{T} = \mathcal{O}(1/ n)$ in the \CV schemes that we consider, the latter bound is sufficient to obtain a  sanity check bound. However, in that proof, the term
$\abs{\widehat{
    \mathcal{R}}_{CV,\alpha}(\Psi_{\alpha},V_{1:K})-\mathcal{R}_{CV,\alpha}(\Psi_{\alpha},V_{1:K})}$
is upper bounded by the sum $\DevExt + \DevCV $  defined in~(\ref{"dev_ext"}) and~(\ref{"dev_CV"}). The probability upper bound for the  latter term 
involves a term of order $\mathcal{O}({1/\sqrt{n_{V}\alpha}})$,  
see (\ref{eq:C-CV-bound}) in the supplement, which is not satisfactory for small $n_V$. 
Therefore one needs an alternative control for $\abs{\widehat{\mathcal{R}}_{CV,\alpha}(\Psi_{\alpha},V_{1:K})-\mathcal{R}_{CV,\alpha}(\Psi_{\alpha},V_{1:K})}$. 
The main ingredient to proceed is the following Markov-type inequality
         
\begin{align}
\PP(\rcvEx(\Psi_{\alpha},V_{1:K})- &\mathcal{R}_{CV,\alpha}(\Psi_{\alpha},V_{1:K}) \geq t ) 
                                      \nonumber \\
\le  \;  &\ffrac{\EE\big(\abs{\hat\risk_\alpha(\Psi_\alpha(S_n),S_n)-
  \risk_{\alpha}(\Psi_\alpha(S_n))}\big)}{t} 
  +\ffrac{\EE(\DevExt+\BiasCV)}{t}, \label{eq:keyTheorem2fromLemma}
\end{align}
which holds true under the stipulated assumptions. The proof is deferred to the supplement (Lemma~\ref{"RCV-markov"}).

It is shown in Section~\ref{"main-theo-lpo-proof"} from the supplement that $\EE(\BiasCV)$ and $\EE(\DevExt)$ are both upper bounded by  $\mathcal{O}(1/\sqrt{n_T\alpha})$ (inequalities (\ref{"exp_c_ext"},\ref{"exp_c_train"})). In addition, the probability upper bound on the supremum deviations on the rare region (Lemma~\ref{lemma:inv-goix} also used in the proof of Theorem~\ref{theo:main-sanity-rare}) shows that the latter  quantity is sub-Gaussian, which yields (\cite{vershynin_2018}, Proposition 2.5.2)  an upper bound  for $\EE\big(\abs{\hat\risk_\alpha(\Psi_\alpha(S_n),S_n)- \risk_{\alpha}(\Psi_\alpha(S_n)}\big)$ of the same order of magnitude as the other terms in the right-hand side (\rhs in the sequel)  of (\ref{eq:keyTheorem2fromLemma}).

The final step of the proof is to derive a probability  upper bound for the opposite of the \textit{l.h.s.} of (\ref{eq:keyTheorem2fromLemma}), that is
$\mathcal{R}_{CV,\alpha}(\Psi_{\alpha},V_{1:K}) -
\rcvEx(\Psi_{\alpha},V_{1:K})$. We use the fact (proved in
Lemma~\ref{"RCV-ge-RN"}) that the \CV risk estimate
$\rcvEx(\Psi_{\alpha}, V_{1:k})$ is always larger than  the empirical risk $\ER_\alpha$
evaluated on its minimizer $\Psi_\alpha(S_n)$, thus
\begin{align}
  \mathcal{R}_{CV,\alpha}(\Psi_{\alpha},V_{1:K}) - \rcvEx(\Psi_{\alpha},V_{1:K}) 
  & \le \mathcal{R}_{CV,\alpha}(\Psi_{\alpha},V_{1:K})-\widehat{ \mathcal{R}}_{\alpha}(\Psi_{\alpha}(S_n),S_n) \nonumber\\
  & \le  \BiasCV +
    | \risk_\alpha(\Psi(S_n)) - \widehat{ \mathcal{R}}_{\alpha}(\Psi_{\alpha}(S_n),S_n)|,  
    \label{eq:decomposeNegativeDeviation-lpo}
\end{align}
where the last inequality follows from the definition of $\BiasCV$ in~(\ref{"bias_CV"}) and the triangle inequality. 
From the proof of Theorem~\ref{theo:main-sanity-rare}, $\BiasCV$  admits a  probability upper bound involving only $n$ and $n_T$ (see~(\ref{"C_train_bound"}) and~(\ref{eq:Q-def}).  
The second term in the \rhs  of (\ref{eq:decomposeNegativeDeviation-lpo}) is less than the supremum deviations of the empirical risk $\ER_\alpha$, which shares the same property (Lemma~\ref{lemma:inv-goix}). Adding up the upper bounds for each term  of the \rhs  of (\ref{eq:keyTheorem2fromLemma}) and~(\ref{eq:decomposeNegativeDeviation-lpo}) concludes the proof, see Section~\ref{"main-theo-lpo-proof"} in the supplement for details. 
\end{proof}

Using Theorem \ref{theo:main-sanity-rare-lpo} and following the same steps as in the proof  of Corollary~\ref{coro:K-fold-sanity-check}, we obtain a sanity-check guarantee regarding leave-$p$-out estimates.
\begin{corollary}[leave-$p$-out sanity check for rare events]\label{coro:lpo-sanity}
Under the assumptions of Theorem~\ref{theo:main-sanity-rare-lpo}, the \lpo \CV  estimate for the conditional risk (\ref{eq:def-risk-extreme}) satisfies with probability $1-17\delta$,
\begin{align*}
	|\widehat{ \mathcal{R}}_{\CV,\alpha}(\Psi_{\alpha},V_{1:K}^{\mathrm{lpo}})-\mathcal{R}_{\alpha}\big(\Psi_{\alpha}(S_n)\big) | \leq E_{lpo}(n,p,\alpha)
	+\ffrac{1}{\delta\sqrt{(n-p)\alpha}}(5M\sqrt{\mathcal{V}_{\mathcal{G}}}+M_5),
\end{align*}
	with 
	\[E_{lpo}(n,p,\alpha)=9M\sqrt{\ffrac{\mathcal{V}_\mathcal{G}}{ (n-p)\alpha}}+\ffrac{9}{(n-p)\alpha} \: .  \]
\end{corollary}

\begin{discussion*}
As it is the case in Section~\ref{sec:sanity-exponential}, our polynomial bounds from Theorem~\ref{theo:main-sanity-rare-lpo} and Corollary~\ref{coro:lpo-sanity} are of the same nature as  the state-of-the-art  for $\alpha=1$, that is  \cite{kearns1999algorithmic}, Theorem~4.2,  and \cite{cornec17},~Proposition~4.3. Again we improve by a factor $\sqrt{\alpha}$ upon the naive method dividing existing bounds by a factor $\alpha$. Concerning the presence of unknown constants, see Remark~\ref{rem:unknownConstants}. In addition to covering the case of rare events, our results extend those of the latter reference in several directions, namely they encompass the \lpo\ scheme whereas \cite{kearns1999algorithmic} only consider the \loo, and they apply to any bounded cost function, not only the Hamming loss. Also  the organisation of our proof is different, for example the risk decomposition~(\ref{"basic-inequality"}) is new.  
\end{discussion*}

\begin{remark}[Tightness of the polynomial bound]\label{rem:tightnessPolynomial}
        A natural question to ask is whether or not the polynomial rate (\wrt the probability $\delta$) is tight concerning the \textit{l.p.o.} \CV scheme. The answer is yes, in the \ERM context, in the general case (that is with a classical risk function and $\alpha=1$). Indeed 
        \cite{kearns1999algorithmic}  show that, without further assumptions on 
        the algorithm $\Psi$  and the cost function $c$, the bound $1/\delta$ can be attained. We conjecture that the same is true for $\alpha < 1$, leaving this question for further work. 
\end{remark}

\begin{remark}[Comparison between the bounds from Theorems~\ref{theo:main-sanity-rare} and~\ref{theo:main-sanity-rare-lpo}]
  Although  Theorem~\ref{theo:main-sanity-rare-lpo} also applies to the $K$-fold, the bound provided by Theorem~\ref{theo:main-sanity-rare} is sharper for this particular \CV scheme for  small values of $\delta$ due to its exponential nature. In other words Theorem~\ref{theo:main-sanity-rare-lpo} has a  greater level of generality than Theorem~\ref{theo:main-sanity-rare} because the upper bound in the latter involved $n_V$, contrarily to the former. The price to pay is a slower tail decay  (polynomial versus exponential).
\end{remark}


\section{Application to  logistic-\Lasso regression}\label{sec:applications}
We now turn to an application of our results to high dimensional
classification as introduced in
Section~\ref{sec:exampleLogisticRegression}. Recall from~\eqref{def:lasso-constraine} that the learning
rule for fixed constraint level $u>0$ takes the form 
\begin{equation}
  \label{eq:Algo_logistic}
  \Psi_{\alpha,u}(S)= \argmin_{\beta\in \mathcal{B}_u}\frac{1}{\alpha n }
\sum_{i\in S}c\left(g_\beta, (X_i,Y_i) \right)\indic{ \| X_{i}\| > \| X_{(\lfloor n\alpha \rfloor )}\| }.
\end{equation}

  Recall also the logistic loss with angular discrimination function,  $ c(g_\beta,(x,y)) = \log( 1 + \exp( - \beta^\top \theta(x)  y)) $. In practice  the aim of the parameter selection procedure is to choose the `best' parameter $u^*$  within a  finite grid  $\mathcal{U} \subset \rset^+$,  regarding the risk of the associated learning rule $\Psi_{\alpha,u}$, that is 
\[ u^*=\argmin_{t\in \mathcal{U}} \risk_{\alpha}\left(\Psi_{\alpha,u}\left(S_n\right)\right) . \]

In view of the exponential nature of the upper bound for \Kfold\ \CV obtained in Section~\ref{sec:sanity-exponential} compared to the polynomial bound for \lpo\ \CV\  (Section~\ref{sec:sanity-polynomial}),  and because \Kfold\ is computationally faster than \lpo\,  we consider a selection procedure based on  a \Kfold\ \CV\ estimate of $\risk_{\alpha}\left(\Psi_{\alpha,u}\left(S_n\right)\right)$, 

$$\hat u=\argmin_{u \in \mathcal{U}}\rCV\left(\Psi_{\alpha,u},V_{1:K}^{\Kfold}\right).$$
We obtain  in Lemma~\ref{lemma: model-selec-extreme} an  upper bound in probability for  the excess risk $\risk\left(\Psi_{\alpha,\hat u}(S_n)\right) - \risk\left(\Psi_{\alpha,u^*}(S_n)\right) $.
Since this upper bound converges to $0$ as $\alpha \to 0$ with $\alpha n \to\infty$, our result ensures in particular the consistency of the selection procedure in extreme regions.

\begin{lemma}\label{lemma: model-selec-extreme}
  With the logistic-LASSO learning rule~(\ref{eq:Algo_logistic}), when selecting the threshold $\hat u$ by \Kfold \CV over the grid $\mathcal{U}$, 
  the excess risk $\risk_\alpha(\Psi_{\alpha,\hat u}(S_n))-\risk_\alpha(\Psi_{\alpha,u^*}(S_n))$ verifies, with probability $1-15\delta$,
\begin{align*}
  &\TR_\alpha  \big(\Psi_{\alpha,\hat u}\left(S_{n}\right)\big)
    -\TR_\alpha \left(\Psi_{\alpha, u^*}\left(S_{n}\right)\right)  \le \\
  &\max(\mathcal{U})\bigg[ 
    2E_{\Kfold}(n,K,\alpha)
    + \frac{40}{3 n\alpha} \log\left(\frac{card(\mathcal{U})}{\delta}\right)   +  40\sqrt{\frac{2}{n\alpha} \log\left(\frac{card(\mathcal{U})}{\delta}\right)}\bigg],
\end{align*}
	with 
	\begin{equation*}
        E_{\Kfold}(n,K,\alpha)=  5M\sqrt{\frac{(d+1)K}{n\alpha}}  +\ffrac{5K}{(K-1)n\alpha},
	\end{equation*}
	for some universal constant  $M$. 
\end{lemma}
\begin{proof}
	By definition of $\hat u$, one has 
	\[\rCV\left(\Psi_{\alpha,\hat u},V_{1:K}^{\Kfold}\right)\leq \rCV\left(\Psi_{\alpha,u^*},V_{1:K}^{\Kfold}\right).\]
	It follows that 
	\begin{align}\label{ineq:feature-select-uniform}
		\TR\left(\Psi_{\alpha,\hat u}\left(S_{n}\right)\right)&-\TR\left(\Psi_{\alpha, u^*}\left(S_{n}\right)\right) \leq \TR\left(\Psi_{\alpha,\hat u}\left(S_{n}\right)\right)- \rCV\left(\Psi_{\alpha,\hat u},V_{1:K}^{\Kfold}\right)\nonumber\\
		&\hspace{8mm} +\rCV\left(\Psi_{\alpha, t^*},V_{1:K}^{\Kfold}\right)-\TR\left(\Psi_{\alpha, t^*}\left(S_{n}\right)\right) \nonumber\\
		&\leq 2\sup_{u \in \mathcal{U}}\left|\rCV\left(\Psi_{\alpha, u}\left(S_{n}\right),V_{1:K}^{\Kfold}\right)-\TR\left(\Psi_{\alpha, u}\left(S_{n}\right)\right)\right|.
	\end{align}
        For fixed $u\in \mathcal{U}$, all the required assumptions of Theorem~\ref{theo:main-sanity-rare} are met, except that  the cost function is not bounded by $1$ but rather by $t$, thus also by $\max(\mathcal{U})$. As explained in Remark~\ref{rem:bounded} our results still apply, up to multiplication of all upper bounds by a factor $\max(\mathcal{U})$. We may thus use Corollary~\ref{coro:K-fold-sanity-check} with $\mathcal{V}_{\mathcal{G}}=d+1$, so that  with probability $1-15\delta$,
		\begin{align*}
	\bigg|\widehat{\mathcal{R}}_{CV,\alpha}(\Psi_{\alpha,u},V_{1:K}^\Kfold)-
                  &\mathcal{R}_{\alpha}\big(\Psi_{\alpha,u}(S_n)\big)\bigg|  \leq   \nonumber \\
                  &\max(\mathcal{U})\bigg[E_{\Kfold}(n,K,\alpha)  
	+ \frac{20}{3 n\alpha} \log\left(\frac{1}{\delta}\right)   +  20\sqrt{\frac{2}{n\alpha} \log\left(\frac{1}{\delta}\right)} \bigg],
	\end{align*}
	where 
	\begin{align*}
          E_{\Kfold}(n,K,\alpha)
          =5 M 
          \sqrt{\frac{(d+1)K}{n\alpha}}  +\ffrac{5K}{(K-1)n\alpha}, 
	\end{align*}
        and where  $M>0$ 
        is a universal constant. 
        We obtain by a union bound, with probability $1-15\delta$,
	\begin{align}\label{ineq:union-bound}
        & \sup_{u \in \mathcal{U}}\bigg|\widehat{\mathcal{R}}_{CV,\alpha}(\Psi_{\alpha,u},
          V_{1:K}^\Kfold) -\mathcal{R}_{\alpha}\big(\Psi_{\alpha,u}(S_n)\big)\bigg|  \leq \nonumber\\
          &   \max(\mathcal{U})\bigg[
            E_{\Kfold}(n,K,\alpha)
	+ \frac{20}{3 n\alpha} \log\left(\frac{card(\mathcal{U})}{\delta}\right)   +  20\sqrt{\frac{2}{n\alpha} \log\left(\frac{card(\mathcal{U})}{\delta}\right)} \bigg].
	\end{align}
	Combining  inequalities \eqref{ineq:feature-select-uniform} and \eqref{ineq:union-bound}  yields the desired result.	
\end{proof}

\section{Numerical Experiments}\label{sec:expes}
The aim of our experiments is to illustrate the tightness of our bounds. The question we ask is whether the error (\emph{resp.} excess risk) upper bound of order  $\mathcal{O}(1/\sqrt{n\alpha})$ describes accurately the behaviour of the \CV\ error  (\emph{resp.} excess risk).  Note that the problem of obtaining lower bounds for the generalization risk of classification algorithms in extreme regions remains to this date an open question in the statistical learning literature dedicated to extremes.
For simplicity we limit our experiments to the \Kfold\ scheme with $K=10$. 

\subsection{\CV error for risk estimation}~

{\bf Experimental setting.} 
We consider the simple setting of a one dimensional threshold-based  classifier ($\mathcal{G}=\{\sign(X-\delta) \mid \delta \in \mathbb{R}\}$)  minimizing the Hamming loss $l(g,(X,Y))=\un{g(X)\neq Y}$.
We investigate the risk estimation error of the \CV\ estimator
$\rcvEx(\Psi_\alpha,V_{1:K})$ defined in~(\ref{"rcvEx-def"}) for
several values of $\alpha$ within the range $\left[1\%,
  20\%\right]$. 
In practice we compute $\rcvEx$ using a  dataset  $\DD_n$,
of size $n=2.10^{4}$ and evaluate the generalization risk of the trained rule
$\Psi_\alpha(S_n)$ on a test set ($\mathcal{D}_{\mathit{Test}}$, of size
$n_{test}=2.10^6$). We perform $n_{simu}=10^4$ experiments  and  we report the average  and the upper $0.90$ quantile of the  abolute error obtained over the $n_{simu}$ experiments. In other words we monitor the 
absolute generalization gap  $\big|\rcvEx(\Psi_\alpha,V_{1:K})-\TR_\alpha(\Psi_\alpha(S_n))\big|$ approximated by the quantity
$\big|\rcvEx(\Psi_\alpha,V_{1:K})-\ER_\alpha(\Psi_\alpha(\DD_n),\DD_{\mathit{Test}}) \big|$
and we report a Monte-Carlo approximation of its expected value 
and its quantile of order $0.90$ for different value of $\alpha$.

{\bf Datasets. } We generate a balanced binary classification dataset $Z_i=(X_i,Y_i) \in \mathcal{Z}=\mathbb{R}\times \{0,1\}$ with $\PP(Y=0)=\PP(Y=1)=1/2$. Both classes are sampled from a $t$-student distribution, with respective parameters $(\mu_i,\sigma_i,\nu_{i}), i=0,1$. 
We set $\mu_0=-\mu_1=1$, $\sigma_0=\frac{3}{5}$, $\sigma_1=3 $,  and $\nu_1=\nu_2=1.5$.
\subsection{\CV excess risk for model selection}
We now describe the empirical analysis of the model selection upper bound presented in Lemma \ref{lemma: model-selec-extreme}.

{\bf Experimental setting.}  We consider the problem of tuning the penalty parameter of a \Lasso logistic regression model. Note that, instead of using the constrained formulation of the \Lasso (cf. Equation (\ref{def:lasso-constraine})) , we consider in our experiments the Lagrangian formulation:
\begin{equation*}
\Psi_{\alpha,\lambda}({S})=\argmin_{\beta\in \mathbb{R}^d}\frac{1}{\alpha n_S} \sum_{i\in S}\left(c\left(\beta^T \Theta\left(X_i\right),Y_i\right)+\lambda\lVert \beta \rVert_1 \right)\indic{\lVert X_{i}\rVert>\lVert X_{(\lfloor \alpha n \rfloor)}\rVert } , 
\end{equation*}
with penality parameter  $\lambda  $ ranging in a finite logarithmic grid 
$$\Delta=\big\{10^{i/30}-1 \mid i\in \{ 1,\ldots, 30 \} \big\}.$$
The reason for using the penalized formulation in practice is mainly a computational one: the latter version can be solved by many standard optimization algorithms (stochastic gradient descent for instance) contrarily to the constrained one that requires special and time consuming optimization routines (see \eg\cite{lee2006efficient,homrighausen2017risk}).  Notice that we leave a gap between theory and practice to be filled in further work. Indeed, analyzing the penalized \Lasso requires different proof techniques and more assumptions. For example, \cite{homrighausen2017risk} work under a  realizability assumption while \cite{chetverikov2021cross}  make some moment assumptions.\\
In the sequel, we study the excess risk of the model selected by \Kfold cross-validation  $\risk_\alpha(\Psi_{\alpha,\hat \lambda}(S_n))-\risk_\alpha(\Psi_{\alpha,\lambda^*}(S_n))$ for several values of $\alpha$ within the range $\left[1\%,10\%\right]$. Similarly to the previous experiment we select $\hat \lambda$ using a dataset $\DD_n$ of size $n=10^4$, then we use a  test set $\DD_{test}$ of size $n_{test}=10^6$ to estimate $\risk_\alpha$ and choose $\lambda^*$ accordingly . Finally, we report the average model selection excess risk and its corresponding $0.9$ quantile for different values of $\alpha$ over $n_{simu}=10^4$ Monte Carlo simulations.

{\bf Datasets. }We generate a balanced binary classification dataset $Z_i=(X_i,Y_i) \in \mathcal{Z}=\mathbb{R}^{20}\times \{0,1\}$ with $\PP(Y=0)=\PP(Y=1)=1/2$. Both classes are sampled from a $t$-multivariate-student distribution, with respective sparse parameters $(\mu_i,\sigma_i,\nu_{i}), i=0,1$. We set $\mu_0=-\mu_1=(e_5,0,\dots,0)$,  $\sigma_0=\sigma_1=10\mathit{I}_{20}$, $\nu_1=\nu_2=1.5$ and $e_5=(1,\dots,1)$ is a 5 dimensional unit vector. 

\subsection{Results}
Figure
~\ref{fig:kfold}  displays the risk estimation error  of the cross-validation  estimator $\rcvEx(\Psi_\alpha,V_{1:K})$ as a function of $\alpha$ on the logarithmic scale. 
As suggested by our theoretical findings, the average error and its quantile  indeed decrease at rate $\mathcal{O}(1/\sqrt{n\alpha})$ as  a function of  $\alpha$. This confirms that our bounds may be sharp up to multiplicative constants.

One must note that in the model selection case (Figure~\ref{fig:excess-risk})  the rate of convergence appears to be faster than $1/\sqrt{n\alpha}$ for values of $n\alpha$ ranging between $500$ and $1000$. This is  not surprising insofar as it corroborates the findings of many  recent works where it is established that the \Lasso algorithm enjoys an \emph{algorithmic stability} \citep{bousquet2002stability} property which induces fast rates for \CV estimates \citep{celisse2016stability,abou-moustafa19a}.  For the smallest values of $n\alpha$ (less than $500$) 
a slower rate is observed.  
This might be explained by the findings of \cite{homrighausen2013lasso} and \cite{Chetverikov2021}  who show that, outside the context of extreme values,   the rate of convergence for cross-validation estimates using $n$ training samples deteriorates as the value of $c=\frac{\ln(d)}{\ln(n)}$ increases.  
\begin{figure}
  \centering
   \includegraphics[scale=0.5]{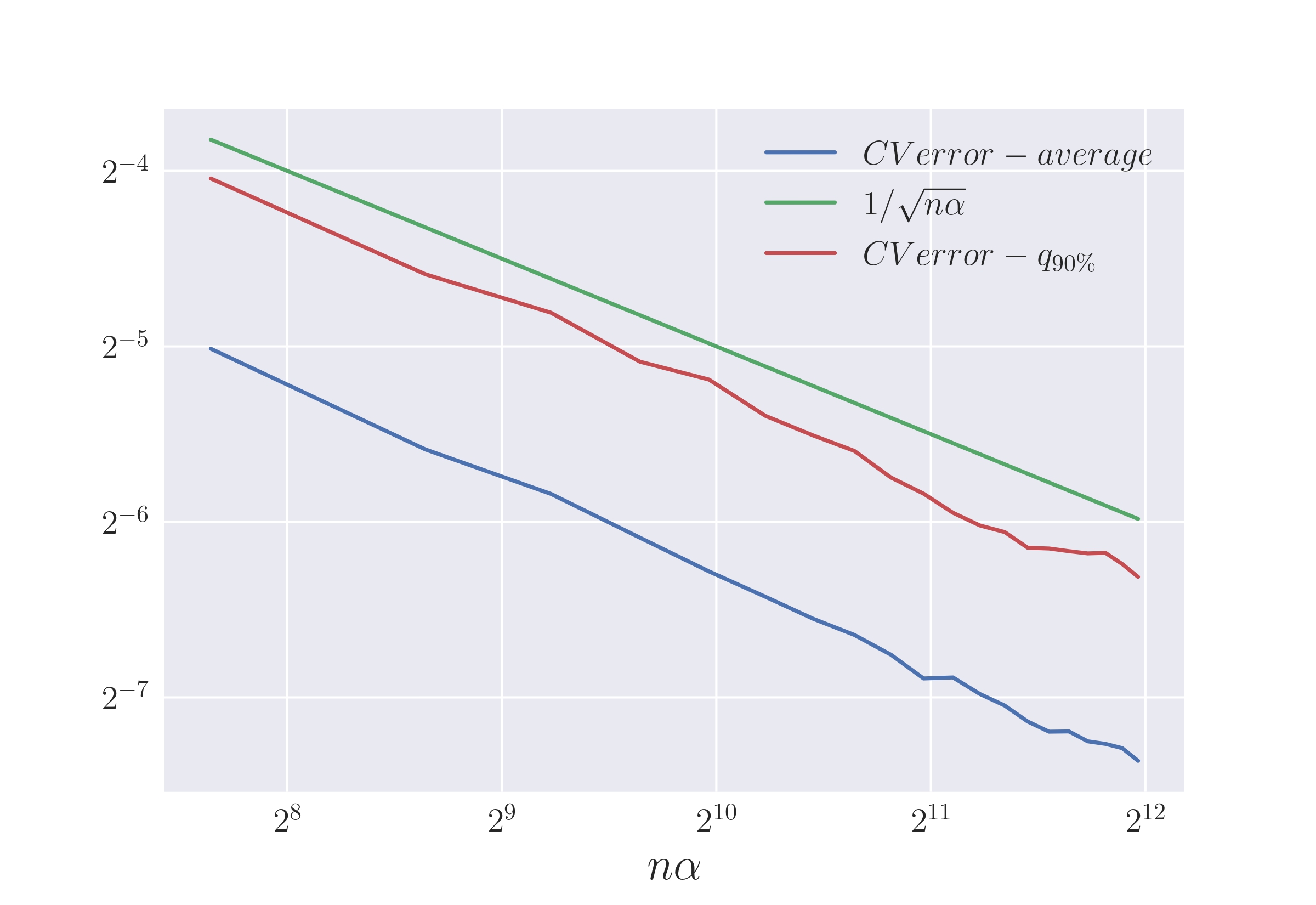}
  \caption{K-fold \CV\ risk estimation absolute error as a function of $n\alpha$ (logarithmic scale): mean and upper quantile at level $0.90$}
  \label{fig:kfold}
\end{figure}

\begin{figure}
	\centering
	\includegraphics[scale=0.5]{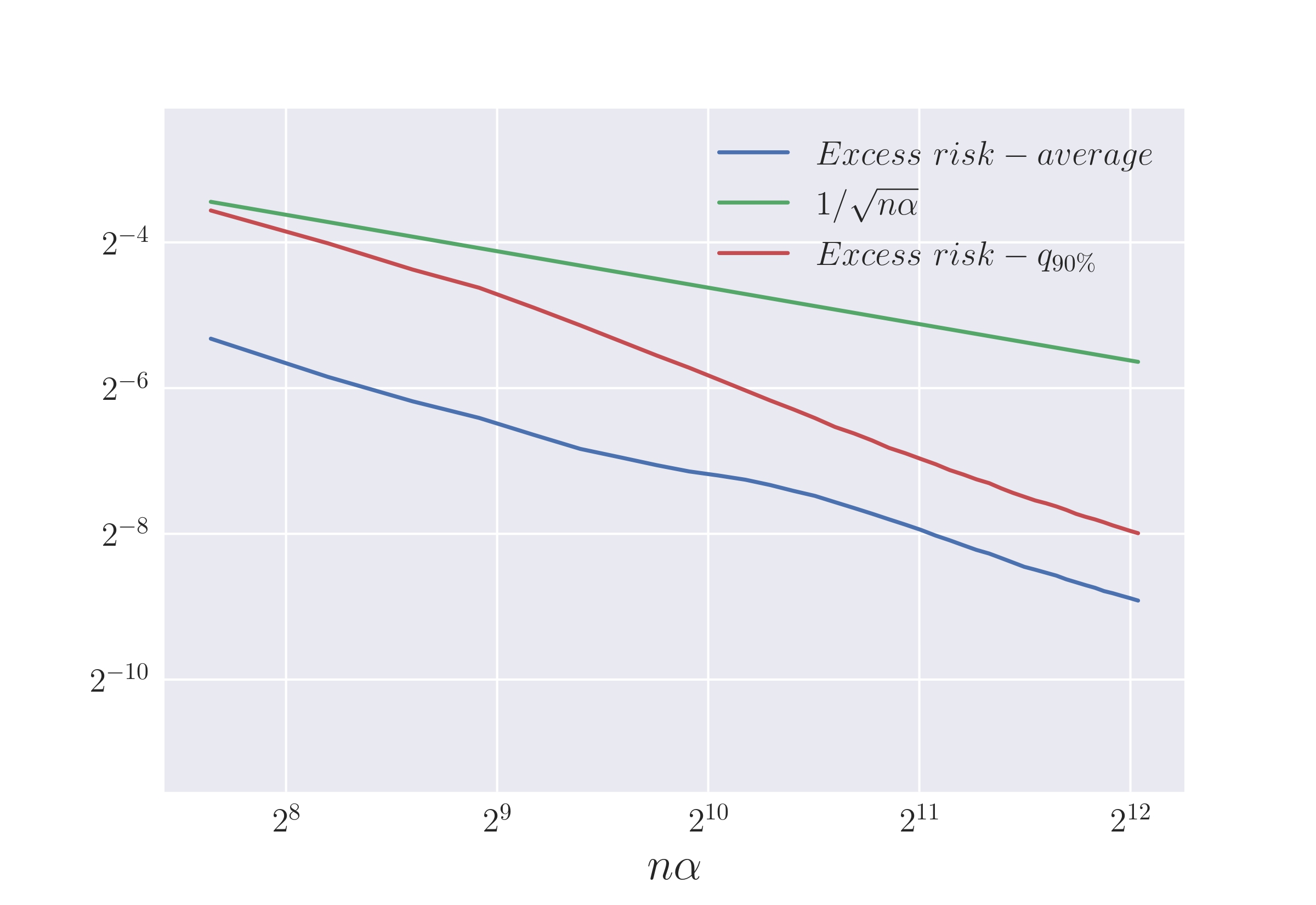}
	\caption{K-fold \CV\ excess risk as a function of $n\alpha$ (logarithmic scale): mean and upper quantile at level $0.90$ }
	\label{fig:excess-risk}
\end{figure}

\section{Discussion}\label{sec:disc}
This paper is a first step towards a theoretical understanding of \CV\ properties for algorithm dedicated to Extreme Value Analysis. It opens several avenues for further research.

First, sanity check bounds are difficult to improve upon without further assumptions regarding  algorithm stability,
\eg loss stability as introduced in \cite{kumar2013near}, where the variance of the \Kfold \CV is shown to be $\mathrm{K}$-times smaller than that of the empirical risk.
A lead for further research  would be to consider such wider classes of
algorithms, for which various stability
hypotheses replace advantageously the assumption of finite \VC
dimension
(see \eg
\cite{kearns1999algorithmic,BousquetElisseeff-2000,bousquet2002stability}). 

Another promising avenue would be to 
relax our assumption of a bounded cost functional in order to extend our results to the setting of extreme quantile regression. Also it would be valuable to  
extend the analysis to  unsupervised problems  which are the main field of application for 
\EVT.  

From a practical perspective, we have restricted the analysis to a
\CV\ strategy where the extreme threshold is chosen randomly as the
$k^{th} = \alpha n^{th}$ order statistics, however this random choice
is made once and for all in the procedure, that is, the same threshold
applies to all folds in the \CV\ scheme, as discussed in Section~\eqref{eq:def-risk-extreme}. 
In other words the threshold
choice is not seen as a part of the learning rule. Using \CV\ to
assess the performance of different threshold choices would require
taking the opposite view and letting the threshold vary across
different folds. How to analyse this alternative \CV\ scheme is left
as an open question.

Finally, 
 even though our work is motivated by the \EVT framework, no regular variation assumption is needed for our results to hold regarding \CV risk estimates, and the precise definition of the rare class $\mathbb{A}$ is unimportant, only the small probability $\alpha$ is. Thus we believe that our results and techniques of proof could
   serve in other statistical learning contexts involving a rare class, such as imbalanced multi-class classification where the learner  must distinguish between several  minority classes.

\section{Amendments}

\subsection{Ethical Approval}
Not applicable.

\subsection{Availability of supporting data}

Not applicable.

\subsection{Authors contribution}
All authors made equal contributions to the analysis, the drafting and critical revision of the manuscript. Each author has approved the final version of the manuscript and agrees to be accountable for all aspects of the work, ensuring its integrity and accuracy.

\subsection{Fundings}
Anne Sabourin’s work has been  partially funded by the ANR projects T-REX and EXSTA. Anass Aghbalou was supported by the chair DSAIDIS. This research has been conducted for Patrice Bertail as part of the project Labex MME-DII (ANR11-LBX-0023- 01).
\subsection{Competing interests}
The authors declare that they have no competing interests.
\subsection{Acknowledgments}
We thank the associate editor and two anonymous referees for their comments which have improved the readability of the paper. 
\bibliographystyle{apalike}
\bibliography{biblio_crossval}
\newpage
\appendix
  \section{Classification on Exreme Covariates: additional results}\label{sec:technicalClassif}


We start-off this section with some necessary background relative to
Classification in Extreme Regions (\cite{jalalzai2018binary})  in
Section~\ref{sec:backgroundClassif}. Next we show in Section~\ref{sec:appliPrediction} that the prediction
problem in multivariate regularly varying vectors described in
Example~\ref{ex:relevanceRalpha-RV} indeed fits in the framework of
\cite{jalalzai2018binary}, thus providing a generic example of the
usefulness for \EVA of the classification
setting 
considered in the present paper.  Finally in Section~\ref{sec:surrogateLoss} we extend the main probabilistic results of \cite{jalalzai2018binary} in order to encompass the setting of Logistic regression considered in Section~\ref{sec:applications} from  the main paper.

\subsection{Background on Classification in Extreme Regions}\label{sec:backgroundClassif}
We  recall here the main notations and assumptions of
\cite{jalalzai2018binary}'s framework. Consider a random pair $(X,Y)$
with $X\in\rset^d$, $Y \in \{1,-1\}$.  The goal of classification in
extreme regions is to minimize (for large values of $t$) the
conditional risk of a classifier $g:\rset^d\mapsto \{-1,1\}$, defined
as $R_t(g ) = \PP[g(X)\neq Y \given \|X\|>t]$ for some norm
$\| \point\|$ on $\rset^d$. More precisely consider the limit superior
at infinity, $R_\infty(g) = \limsup_{t\to\infty} R_\infty(g)$. The
problem may be written as
\begin{equation*}
  \underset{g\in \mathcal{G}}{\mathrm{minimize } }\;  R_\infty(g), 
\end{equation*}
where $\mathcal{G}$ is  a family of measurable functions $\rset^d\to\{-1,+1\}$. 

It is assumed in \cite{jalalzai2018binary}  that each class distribution $\mathcal{L}(X\;|\; Y= \sigma 1\}$ for  $\sigma\in\{+,-\}$ is multivariate regularly varying with same tail index, \ie 
there exist  limit measures $\mu_\sigma$ which are  finite on sets bounded away from $0_{\rset^d}$, such that 
\begin{equation}
  \label{eq:RV_class}
  b(t) \PP[ t^{-1} X \in A \; |\; Y = \sigma 1 ]\xrightarrow[t\to\infty]{} \mu_\sigma(A), 
\end{equation}
for any measurable set $A\subset\rset^d $ such that $0_{\rset^d}\notin\partial(A)$  and $\mu_\sigma(\partial A) =  0$.  Standard references regarding multivariate regular variations include \cite{resnick2013extreme,hult2006regular}. 
In~\cite{jalalzai2018binary} the authors assume that $\alpha=1$ and
$b(t)=t$ to alleviate notations, mainly because this is verified in
particular 
when the  univariate margins of $X$ are standardized to unit Pareto. Also
the covariate $X$ is assumed to have non-negative entries,
$X\in\rset_+^d$ almost surely. However these simplifications are not
used in the proofs which are also valid in the general case considered here, that is if $X$ is allowed to have
non-positive entries and  if $b$ is  a regularly varying function 
with 
index $\alpha>0$, \ie $b(tx)/b(t) \to x^{-\alpha}$ as
$t\to\infty$, for all fixed $x>0$.
The marginal distribution of the random covariate  $X$ is then also regularly varying  with scaling function $b$ and limit measure $\mu = \pi \mu_+ + (1-\pi)\mu_-$, where $\pi = \PP[Y=1]$.
Consider  the probability of the positive class above covariate
level $t$, $\pi_t = \PP[Y= +1 \;|\; \|X\|\ge t]$. 
Denoting by ${\mathcal{B}}$ the unit ball
${\mathcal{B}}= \{x\in \rset^d: \|x\|_l\le 1 \}$, it holds that
$$
\pi_t \xrightarrow[t\to\infty]{}  \pi_\infty =  \pi \frac{\mu_+({\mathcal{B}}^c)}{\mu({\mathcal{B}}^c)}.
$$
It is assumed 
that $\pi_\infty\notin\{0,1\}$ (otherwise the classification problem in the limit is trivial). 

A key probabilistic object in the context of binary classification is 
the Bayes regression function, classically defined (almost surely) as
$\eta(X) = \PP[Y = +1 \; |\; X]$. Then it is a well known fact that
the classifier $g^*(x) = 2\un{\eta(x)\ge 1/2} - 1$ 
minimizes the $0$--$1$ loss among all measurable prediction
functions. 
Under the regular variation Assumption~\eqref{eq:RV_class}, 
one may define an `extreme pair'
$(X_\infty,Y_\infty)$ 
on $(\rset^d\setminus{\mathcal{B}})\times \{-1,1\}$ with distribution given
by
\begin{equation}\label{eq:limitpair}
  \PP[X_\infty \in A, Y_\infty = \sigma 1 ]
  = \lim_{t\to\infty}\PP[ t^{-1} X \in A, Y = \sigma 1 \given \|X\|\ge t]
  = \pi_{\infty,\sigma} \frac{\mu_\sigma(A)}{\mu_\sigma({\mathcal{B}}^c)},
\end{equation}
where $\pi_{\infty,+}=\pi_{\infty}$ and $\pi_{\infty,-} = 1-\pi_\infty$, for $A\subset \rset^d\setminus {\mathcal{B}}$ such that
$\mu_\sigma(\partial A)=0$. A crucial quantity in \cite{jalalzai2018binary} is the Bayes regression function $\eta_\infty$ relative to the extreme pair $(X_\infty, Y_\infty)$, defined as $\eta_\infty(X)= \PP[Y_\infty = 1 \given X_\infty]$, almost surely.  It is shown in the latter reference that $\eta_\infty$ depends on $x$ only through the angular component $\theta(x)$, \ie $\eta_\infty(tx) = \eta_\infty(x)$ for all $t\ge 0$. This fundamental observations is the main heuristic justification behind the \ERM algorithm they propose, which consists in minimizing an empirical risk over a class of \emph{angular} predictors of the kind $g(x) = h\circ\theta(x)$.
To show that this strategy indeed yields a predictor with good generalization performance in terms of $R_\infty$,  a key assumption is the following
\begin{condition}\label{assum:uniform-convergence}[Assumption 2 in \cite{jalalzai2018binary}]
  The regression function of extremes $\eta_\infty$ is continuous on
  the unit sphere $S=\left\{\omega\in \rset^d\;:\;  \|\omega\|=1\right\}$
  and
  $$\sup _{\omega \in S}\left|\eta(t \omega)-\eta_{\infty}(\omega)\right|
  \underset{t \rightarrow \infty}{\longrightarrow} 0,$$
  where $\eta(x)=\PP(Y=1\mid X=x)$ is the standard regression
  function.
\end{condition}
Under Condition~\ref{assum:uniform-convergence} it is shown (Theorem 1 in \cite{jalalzai2018binary}) that $R_\infty$ is minimized by a  classifier based on $\eta_\infty$, namely $g_\infty^* (x)= \un{\eta_\infty^*(\theta(x))\ge 1/2}$. We shall show in Section~\ref{sec:surrogateLoss} that this key property is preserved when considering the logistic loss instead of the $0$--$1$ loss. For now we investigate in the next section how this setting may be applied to the problem of predicting the occurrence of an extreme value of a missing component, given that the other components are large.

\subsection{Application to prediction in regularly varying vectors}\label{sec:appliPrediction}
We now turn to the context envisioned in Example~\ref{ex:relevanceRalpha-RV}  from the main paper. Recall that   $Z = ( Z_1,\ldots, Z_{d+1})\in\rset^{d+1}$ is  a random vector and let, for some $c\in(0,1)$,
\begin{equation}\label{eq:defineXY-prediction}
  X  = (Z_1,\ldots, Z_{d}) \text{ and }
  Y = \un{ \frac{Z_{d+1}}{\| (Z_1,\ldots, Z_{d+1}) \| }  > c } = \un{ \theta(Z)_{d+1} >c}, 
  \end{equation}
where $\|\point\|$ is the  $\ell^p$ norm on $\rset^d$ for some $p\ge 1$ and where we recall $\theta(z) = \|z\|^{-1} z$. Assume that $Z$ has a continuous density $p$ with respect to the Lebesgue measure on $\rset^{d+1}$. Let us assume 
that regular variation of the density of $Z$ holds (\cite{de1987regular,Cai2011}), \ie for some positive function $q: \rset^{d+1}\to\rset_+$,   for some regularly varying function  $b$ on $\rset_+$ with index $\alpha>0$ (\ie $b(tx)/b(t)\to x^{-\alpha}$ for fixed $x>0$), 
\begin{equation}\label{eq:regvarDensity}
 \sup_{\omega \in\sphere_{d+1}} 
      | t^{d+1}b(t) p(t\omega)   - q(\omega) | 
      \xrightarrow[t\to\infty]{} 0,
\end{equation}
where $q$ is a function defined on $\rset^{d+1}\setminus\{0_{\rset^{d+1}}\}$ and $\sphere_{d+1}$ is the unit sphere in $\rset^{d+1}$. Condition~\eqref{eq:regvarDensity} implies that  $Z$ is regularly varying with index $\alpha$, scaling function $b$  and limit measure $\mu_Z$ which admits $q$ as a density. Also the density $q$ is homogeneous,  $q(tx) = t^{-d-1-\alpha} q(x)$. Finally, a useful property is that if~\eqref{eq:regvarDensity} holds then also 
\begin{equation}
    \label{eq:cvDensitybis}
    \sup_{\lVert z \rVert \ge 1} 
      | p(t z) t^{d+1}b(t)  - q(z) | 
      \xrightarrow[t\to\infty]{} 0, 
    \end{equation}
 see  \cite{de1987regular} and \cite{Cai2011} for proofs. 

The following result ensures that  the \ERM strategy proposed in~\cite{jalalzai2018binary}, which is the leading example of the present paper,  enjoys sound theoretical guarantees in the prediction setting considered here.   A similar result has been obtained recently in~\cite{huet2023regression} for regression on extreme regions with a continous target, as mentioned in Example~\ref{ex:relevanceRalpha-RV} from the main paper.  
Although our  proof bears strong similarities with the argument developed in~\cite{huet2023regression} in the regression setting, we give the full details here in the classification framework for the sake of completeness. 

\begin{proposition}\label{prop:regvardensity}
 If regular variation of densities~\eqref{eq:regvarDensity} holds, then 
the pair $(X,Y)$ defined in~\eqref{eq:defineXY-prediction} satisfies \cite{jalalzai2018binary}'s assumption, namely regular variation of class distributions~\eqref{eq:RV_class} and Condition~\ref{assum:uniform-convergence}. 
\end{proposition}
\begin{proof}
  For conciseness in the sequel, for $z\in\rset^{d+1}$, we write $z_{1:d}$ for the vector $(z_1,\ldots, z_d)$. Also 
  we recall $\theta(z) = \|z\|^{-1}z$. Also we can and will asssume that $b(t)$ in \eqref{eq:regvarDensity} is chosen as $b(t) = \PP[\|Z\|>t]^{-1}$. 
  \paragraph{Proof of~\eqref{eq:RV_class}}
  For any bounded, continuous function $\psi: \rset^d\to \rset$ which vanishes on a neighborhood of $0$, denoting by $\tilde \psi$ the function defined on $\rset^{d+1}$ as $\tilde\psi(z) = \psi(z_{1:d})$ we have 
\begin{align*}
  b(t) \EE[ \psi(t^{-1} X) \given Y=1 ]
  & = b(t) \EE\Big[\underbrace{ \tilde \psi(t^{-1} Z) \un{\theta(t^{-1}Z)_{d+1}} >c }_{\psi'(t^{-1}Z)}\Big]
    / \PP[\theta(Z)_{d+1} >c].
\end{align*}
The set of discontinuity points of $\psi'$ is
$\mathcal{D} = \{z\in\rset^{d+1}: \theta(z)_{d+1} = c\}$. Because
$\mu_Z$ has a density $q$, we have $\mu_Z(\mathcal{D})=0$ whence the
above display yields, with $\pi  = \PP[Y=1] = \PP[\theta(Z)_{d+1} >c]$,
\begin{align*}
  b(t) \EE[ \psi(t^{-1} X) \given Y=1 ]\xrightarrow[t\to\infty]{}
  \pi^{-1} \int_{\rset^{d+1}\setminus\{0\}} \psi(z) \un{\theta(z)_{d+1}> c} \ud\mu_Z(z). 
\end{align*}
Thus~\eqref{eq:RV_class} holds with $b(t) = \PP[\|Z\|>t]$,
$\mu_+(A) =\pi^{-1} \mu_Z(A\cap \mathcal{C}_+)$ where $\mathcal{C}_+$ is the cone
$\mathcal{C}_+ = \{z\in\rset^{d+1}\setminus\{0\}, \theta(z)_{d+1}>c \}$. The same argument for the negative class shows that $\mu_-(A) = (1-\pi)^{-1}\mu_Z(A\cap \mathcal{C}_-)$ where 
$\mathcal{C}_- = \{z\in\rset^{d+1}\setminus\{0\}, \theta(z)_{d+1}\le c \}$.
 \paragraph{Proof that Condition~\ref{assum:uniform-convergence} holds true}
We have for $x\in\rset^d\setminus\{0_{\rset^d}\}$,
\begin{align}
  \eta(x) & = \PP[Y=1|X=x] = \PP[\theta(Z)_{d+1}>c \given Z_{1:d} =x]\nonumber \\
  & = \int_\rset\un{\theta(x, s)_{d+1} >c} \frac{p(x,s)}{p(x)} \ud s \label{eq:regfun}
\end{align}
where we write for simplicity $p(x,s)$  for the density of $Z$ at point $(x,s)$ and $p(x)$ for the marginal density of $Z_{1:d}$ at $x$.  

On the other hand, let $Z_\infty$ be a random vector in $\rset^{d+1}$ distributed according to the restriction of $\mu_Z(\point)$ to ${\mathcal{B}}_{d+1}^c$, where ${\mathcal{B}}_{r}$ is the $r$-dimensional unit ball for $r\ge 2$. Then $Z_\infty$ has density function $q$.  The argument of the latter paragraph shows that the limit pair in~\eqref{eq:limitpair} has distribution on ${\mathcal{B}}_d^c\times\{-1,1\}$  given  by 
$$\mathcal{L}(X_\infty, Y_\infty) \eqd  \mathcal{L}(Z_{\infty,1:d},  2 \un{ \theta(Z_{\infty})_{d+1}>c }-1 \;|\; \|Z_{\infty,1:d}\|\ge 1). $$
Thus the regression function for the extreme pair is, for $x\in{\mathcal{B}}_d^c$, 
\begin{align}
  \eta_\infty(x) &=  \PP[\theta(Z_\infty)_{d+1}>c  \given Z_{\infty, 1:d} =x  ]\nonumber \\
                 & = \int_\rset\un{\theta(x, s)_{d+1} >c} \frac{q(x,s)}{q(x)} \ud s
                   \label{eq:regfun_ext}
\end{align}
where for $x\in\rset^d\setminus\{0_{\rset^d}\}$, we denote $ q(x) =  \int_\rset q(x,s)\ud s$. Notice that for such $x$, and for 
$s\in\rset$,
$t>0$,  we have $q(tx,ts) = t^{-d-1-\alpha} q(x,s)$ whereas  $q(tx) = t^{-d-\alpha} q(x)$. Thus we have immediatly, for $t>0$ and $x\in\rset^d\setminus\{0\}$, 
\begin{align*}
  \eta_\infty(tx)&  = \int_\rset\un{\theta(tx, s)_{d+1} >c} \frac{q(tx,s)}{q(tx)} \ud s \\
                 & =   \int_\rset\un{\theta(tx, ts')_{d+1} >c} \frac{q(tx,ts')}{q(tx)} t\ud s' \\
                 & =\int_\rset\un{\theta(x, s')_{d+1} >c} \frac{q(x,s')}{t q(x)} t\ud s' \\
  & = \eta_\infty(x).
\end{align*}

Combining~\eqref{eq:regfun} and~\eqref{eq:regfun_ext} we obtain  that for all $\omega\in S_d$ and $t\ge 1$,  by a change of variables, 
\begin{align}
  |\eta(t\omega) -\eta_\infty(\omega)|
  &  =  |\eta(t\omega) -\eta_\infty(t\omega)| \le \int_\rset \left|  \frac{p(t\omega,s)}{p(t\omega)}   -  \frac{q(t\omega,s)}{q(t\omega)}  \right| \ud s \nonumber \\
  & = \int_\rset \left|  \frac{p(t\omega,ts')}{p(t\omega)}   -  \frac{q(t\omega,ts')}{q(t\omega)}  \right| t\ud s' \nonumber \\
  &= \int_\rset \left|  \frac{p(t\omega,ts)}{p(t\omega)}   -  \frac{q(\omega,s)}{t q(\omega)}  \right| t\ud s \nonumber \\
  &= \int_\rset \left|  \frac{p(t\omega,ts)}{t^{-1}p(t\omega)}   -  \frac{q(\omega,s)}{q(\omega)}  \right| \ud s \nonumber \\
    &= \int_\rset \left|  \frac{h(t) p(t\omega,ts)}{t^{-1}h(t) p(t\omega)}   -  \frac{q(\omega,s)}{q(\omega)}  \right| \ud s, 
    \label{eq:boundDiffRegfun}
\end{align}
where we have put $h(t) = t^{d+1} / \PP[\|Z\|>t]$. The rest of the proof 
relies on the following two lemmas, stated and proved in the supplementary material of the cited reference,  which we repeat here for ease of reading.
\begin{lemma}[Uniform Convergence of marginals of $p$]\label{lem:Unifcvmarginal}~
  If the random vector $Z$ satisfies~\eqref{eq:regvarDensity}, 
  we have
  \begin{equation*}
    \begin{gathered}
      \sup_{\omega\in\sphere_d} \left| \int_\rset t^{-1} h(t)  p(t \omega, z) \ud z - q(\omega) \right| \xrightarrow[t\to\infty]{} 0, \quad \text{where } \\
      q( \omega) = \int_{s\in\rset} q(\omega,s) \ud s,  \quad  \text{ and }
      h(t) = t^{d+1}/\PP[\|Z\| \ge t].     
    \end{gathered}
 \end{equation*}
\end{lemma}

\begin{lemma}[Upper and lower bounds for the marginals of $q$]\label{lem:boundsMarginalQ}
  Under the condition that~\eqref{eq:regvarDensity} holds and  that $q$ in~\eqref{eq:regvarDensity} is uniformly lower bounded on the sphere, \ie 
  \begin{equation}
    \label{eq:lowerboundq}
    \inf_{\tilde \omega\in\sphere_{d+1}} q(\tilde \omega)>0, 
  \end{equation}
  there exist positive constants $c,C >0$ such that  for all $\omega\in\sphere_d$,
   \begin{align*}
    c \le  \int q(\omega,z)\ud z \le C.
   \end{align*}
 \end{lemma}
 An additional useful result (also proved in~\cite{huet2023regression}) is that, letting $A(t,s) = \sup_{\omega\in\sphere_{d}} \frac{h(t)}{ h(t\| (x,s)\| )} $ we have for some $t_0\ge 0$, for all $s\in\rset$
 \begin{equation}
   \label{eq:supboundh}
   A(t,s) \le 2 (1+s^p)^{\frac{-d-\alpha/2-1}{p}}. 
 \end{equation}
 We now get back to the proof that Condition~\ref{assum:uniform-convergence} holds true. From~\eqref{eq:boundDiffRegfun} we obtain
 \begin{align}
   \sup_{\omega\in\sphere_d}| \eta(\omega) - \eta_{\infty}(\omega) |
   & \le  \int_\rset \sup_{\omega\in\sphere_d}
     \underbrace{
     \left|  \frac{h(t) p(t\omega,ts)}{t^{-1}h(t) p(t\omega)}   -  \frac{q(\omega,s)}{q(\omega)}  \right| \ud s
     }_{J(t,s)}\label{boundsupJ}
    \end{align}
    For fixed $s$ the integrand $J(t,s)$ in~\eqref{boundsupJ} converges to $0$ as $t\to\infty$. Indeed write
    \begin{align*}
      J(t,s)\le \sup_{\omega\in\sphere_d}  h(t) p(t\omega,ts)\Big|\frac{1}{t^{-1}h(t) p(t\omega)} - \frac{1}{q(\omega)}  \Big| +
      \sup_{\omega\in\sphere_d} \frac{1}{q(\omega)}
      \Big| h(t) p(t\omega,ts) - q(\omega,s)  \Big|.
    \end{align*}
As $t\to\infty$, the first term on the right-hand side goes to zero because of Lemmas~\ref{lem:Unifcvmarginal} and~\ref{lem:boundsMarginalQ}. The second term
goes to zero because of Lemma~\eqref{eq:lowerboundq} and our working hypothesis of regular variation of densities~\eqref{eq:regvarDensity}.

Their remains to show that $J(t,s) $ is upper bounded by an integrable function of $s$ in order to apply dominated convergence. This is done by decomposing $J(t,s)$ as 
   \begin{align*}
     J(t,s) & \le
     \underbrace{ \sup_{\omega\in\sphere_d} \frac{h(t)}{ h(t\lVert (\omega,s)\rVert )}}_{A(t,s)}
      \underbrace{\sup_{\omega\in\sphere_d}
      \frac{ h(t\lVert (\omega,s)\rVert ) p\Big( t\lVert (\omega, s)\rVert  \theta(\omega,s) \Big)
              }{t^{-1}h(t)p(t \omega)}}_{B(t,s)} +  \dots \\
            & \dots   \underbrace{ \sup_{\omega \in \sphere_d}
              \frac{q(\omega,s)}{q(  \omega)}}_{C(t,s)} \\
     &= A(t,s) B(t,s) + C(t,s). 
                 \end{align*}
 From~\eqref{eq:supboundh} we have that for $t$ large enough,  $A(t,s)\le 2 (1+s^p)^{\frac{-d-\alpha/2-1}{p}}$ which is an integrable function of $s$. 
 Also, using~\eqref{eq:cvDensitybis}, both the numerator and the denominator in the definition of
 $B(t,r)$ converge as $t\to +\infty$, uniformly over
 $\omega \in\sphere_d$ and $s\in\rset$, respectively to $q(\omega,s)$
 and $q(\omega)$. Now $q(\omega)$ is uniformly lower bounded 
 (Lemma~\ref{lem:boundsMarginalQ}) and $q( \omega ,r)$ is uniformly upper bounded
 for $\omega\in\sphere_d$ (by homogeneity). Thus, for $t$ large enough, for all $r$, $B(t,r)\le C$  for some constant $C>0$. 
Finally by homogeneity of $q$ and Lemma~\ref{lem:boundsMarginalQ} again, we have
\begin{equation*}
C(t,s) \le \sup_{\omega\in\sphere_d}\lVert  (\omega,s)\rVert ^{-\alpha-d-1} \frac{\max_{\tilde \omega \in \sphere_{d+1}} q(\tilde \omega)}{c} = (1 + s^p)^{\frac{-\alpha-d-1}{p}} \frac{\max_{\tilde \omega\in \sphere_{d+1}} q(\tilde \omega)}{c}.
\end{equation*}
The latter expression is an integrable function of $s$. We have shown that $J(t,s)$ is upper boudned by an integrable function of $s$, which completes the proof of the uniform convergence of $\eta$ in Condition~\ref{assum:uniform-convergence}.

As a last step we prove that $\eta_\infty$ is continuous on $\sphere_{d}$. 
Recall that for  $x\in \rset^d \setminus \{0_{\rset^d}\}$, 
 $ \eta_\infty(x) = \frac{1}{q(x)}\int_\rset \un{
    \frac{s}{\lVert (x,s)\rVert } > c}q(x,s)ds. $
Also for $\omega\in\sphere_d$, because $\|\point\|$ is the $\ell^p$ norm 
we have
\begin{align*}
  \frac{s}{\lVert (\omega,s)\rVert } > c \iff s>c(1-c^p)^{-1/p}:=g(c). 
\end{align*}
Thus for $\omega\in\sphere_d$
\begin{equation*}
  \eta_\infty(\omega) = \frac{1}{q(\omega)}\int_{g(c)}^\infty q(\omega,s)ds. 
\end{equation*}
The continuity of $p$ implies that of $q$ under our assumption of uniform convergence of densities~\eqref{eq:regvarDensity}. 
 By homogeneity of $q$, we have
\begin{equation*}
q(\omega,z) =   \lVert (\omega,s)\rVert ^{-d-\alpha-1}q\big(\theta(\omega,s)\big) \le (1+s^p)^{\frac{-d-\alpha-1}{p}}\max_{\omega \in \sphere_{d+1}}q(\omega).
\end{equation*}
Since $s \mapsto (1+s^p)^{\frac{-d-\alpha-1}{p}}$ is integrable over $\rset$, the dominated convergence theorem for continuity applies,  
proving the continuity of $\eta_\infty$ on $\sphere_d$. 
\end{proof}



\subsection{
 Optimality of angular predictors above asymptotically high levels for  general loss functions}\label{sec:surrogateLoss}

In this section we  extend \cite{jalalzai2018binary}'s main probabilistic result (Theorem~1 in the cited reference) to  general bounded loss functions $c(g,z)$ where $z=(x,y)\in\rset^d\times\{-1,1\}$ and real-valued predictors $g:\rset^d\to\rset$, in order to cover in particular  the case of  logistic regression. Indeed in such a case $c(g(x), y) = \log(1+\exp(-g(x)y))$. 
In  \cite{jalalzai2018binary} the focus is instead on the  $0$--$1$ loss $c(g(x), y) = \un{g(x)y<0}$. Our aim is to show that for generic loss functions, there exists a predictor which depends on the angle only, which minimizes the  risk above  asymptotically large radial levels. 


 As an extension of the binary classification problem detailed in
 Section~\ref{sec:backgroundClassif},
 let 
 $R_{t}(g)=\EE\left[ c(g,Z)\mid \|X\| \geq t \right]$ where $Z=(X,Y)$
 and let
 $R_{\infty}\left(g\right)=\limsup\limits_{t \rightarrow \infty}
 R_{t}\left(g\right)$.  The probabilistic objects
 $\mu_\sigma, \sigma\in\{-,+\}$, $(X_\infty,Y_\infty)$, $\pi,  \pi_t, \pi_\infty$,
 $\eta,\eta_\infty$ are defined as in
 Section~\ref{sec:backgroundClassif} and we assume that Condition~\ref{assum:uniform-convergence} holds true. Recall that the regression function $\eta_\infty$ for the extreme pair $(X_\infty,Y_\infty)$ is constant  along rays, 
 that is $\eta_\infty(tx)=\eta_\infty\left(x\right)$ for all $t\geq 0$, in other words  $\eta_\infty(x)=\eta_\infty\left(\theta\left(x\right)\right)$.

	

We further assume that 
  $$\forall z=(x,y)\in \rset^d\times\{-1,1\},\quad c(g,z)=\phi(g(x)y),$$
 for some function $\phi$,  so that
 \begin{equation}\label{eq:risk-formula}
 	R_t(g)=\EE\left[\eta\left(X\right)\phi\left(g(X)\right)+\left(1-\eta\left(X\right)\right)\phi\left(-g(X)\right) \given \|X\|>t\right].
 \end{equation}
 This is the case for logistic regression with
 $\phi(r) = \log(1+e^{-r})$. Another well known feature of logistic
 regression is that the (standard) risk $R_0(g) = \EE[\phi(g(X)Y)]$ is minimized by the predictor 
 $g^*(x) = a\circ\eta(x)  $ with  $a(r) = \log(r/(1-r))$ (see \emph{e.g.} \cite{zhang2004statistical}).
 As an immediate consequence the prediction function
 \begin{equation}
   \label{eq:defginfty}
g_\infty^* (x) = a\circ\eta_\infty(x) = \log(\eta_\infty(x)/(1-\eta_\infty(x))  )   
 \end{equation}
minimizes the risk for the extreme pair, $\EE[c(g(X_\infty, Y_\infty)) ]$. 
 Notice that the function $a$ is uniformly continuous on any compact interval $[\eta_1,\eta_2]\subset (0,1)$, and so are  the functions $\phi\circ a$ and $\phi\circ(-a)$ on such an interval.

 We may now state the main result of this section, which covers the logistic case considered here. 

\begin{theorem}\label{theorem:classif-extreme-angular}
  Let the function $\phi$ be such that there exists a function $a:[0,1]\to\rset $, such that for any random pair $(X,Y)$,
  the classical risk $R_0$ admits as a minimizer (among all measurable
  functions) a prediction function $g^*$ of the form
  $g^*(x) = a\circ \eta(x) $.  Assume in addition that the functions
  $r\mapsto \phi\circ a(r) $ and $r\mapsto \phi \circ(-a)(r)$ are uniformly
  continuous over $[\eta_1,\eta_2]= \range(\eta) \subset [0,1]$. If the regular
  variation assumption~\eqref{eq:RV_class} and
  Condition~\ref{assum:uniform-convergence} hold true, then the predictor 
  $g_\infty^* = a\circ \eta_\infty$ minimizes the risk at infinity $R_\infty = \limsup R_t$.
\end{theorem}
The practical impact of Theorem~\ref{theorem:classif-extreme-angular}
is that, as in \cite{jalalzai2018binary}, there exists a
discrimination function $g_\infty^*$ that minimizes the asymptotic
risk, and which depends on the angle of the input only.  In the case
of logistic regression, in view of the discussion above the statement,
the assumptions of Theorem~\ref{theorem:classif-extreme-angular} are
met as soon as the regression function $\eta$ is bounded away from $0$
and $1$, that is $\range(\eta)\subset [\eta_1,\eta_2]$ for some
$0<\eta_1\le \eta_2 < 1$.

\begin{proof}[Proof of Theorem~\ref{theorem:classif-extreme-angular}]
  First notice that $R_t(g^*)\leq R_t(g)$ for all measurable function $g$. 
  Taking the limit superior on both sides, we obtain that $g^*$  minimizes $R_\infty$ as well. 
	Thus we only need to show that $R_t(g^*) - R_t(g^*_\infty)\to 0$. 
	Now using~\eqref{eq:risk-formula} and the fact that $g^*(x) = a\circ\eta(x)$ and $g_\infty^*(x) = a\circ\eta_\infty(x)$,  write 
	\begin{align*}
         | R_t(g^*_\infty)-R_t(g^*)|
          &= \Big| \; \EE[ \Delta(X)  \given \|X\|>t ] \; \Big|
             \le \sup_{x\in\rset^d : \|x\|>t} | \Delta(x)|, 
        \end{align*}
        where
        \begin{align*}
          \Delta(x) & = \eta(x)( \phi\circ a \circ \eta_\infty(x) - \phi\circ a \circ \eta(x)  )   +   \dotsb  \\
      &  
        (1-\eta(x)) ( \phi\circ (-a) \circ \eta_\infty(x) - \phi\circ(- a) \circ \eta(x)  ).  \\
        \end{align*}
        Because the latter expression is a convex combination, we have for $x$ such that $\|x\|\ge t$, 
        \begin{align*}
          |\Delta(x) |  \le  \max\Big( |\phi\circ a \circ \eta_\infty(x) - \phi\circ a \circ \eta(x) |, \quad 
    | \phi\circ (-a) \circ \eta_\infty(x) - \phi\circ(- a) \circ \eta(x) |     \Big). 
        \end{align*}
        By Condition~\ref{assum:uniform-convergence} and the assumption in the statement that both $\phi\circ a $ and $\phi\circ(-a)$ are uniformly continuous we get
        $$ \sup_{x: \|x\|>t } |\Delta(x)|\xrightarrow[t\to\infty]{} 0, $$
        and the result follows.  
\end{proof}


	

\begin{remark}[Hinge loss] 
  Even though we only consider logistic regression as an application
  for simplicity, Theorem~\ref{theorem:classif-extreme-angular} also
  applies upon choosing the  hinge loss as a  loss function, 
  $c(g,z) =\phi(g(x)y)= \max\left( 0, 1 - yg(x) \right)$, which is
  typically used in SVM regression. The minimizer of the classical
  risk $R_0$ is then $g^*(x) = \mathrm{sign}(2\eta(x)-1)$ (see \emph{e.g.} \cite{zhang2004statistical,bartlett2006}), thus the conditions of Theorem~\ref{theorem:classif-extreme-angular} regarding  the form $g^*(x) = a\circ\eta(x)$ is met as well as uniform continuity of $\phi\circ a$ and $\phi\circ (-a)$.
\end{remark}


\section{Generic technical tools}\label{sec:genericTools}

We recall the following McDiarmid's extension of Bernstein inequality (Theorem 3.8 in \cite{McDiarmid98conc}).
\begin{proposition}\label{"Bernstein-extension"}
	For a sequence of observations $(Z_1,Z_2,\dots,Z_n) \in \mathcal{Z}^n$ and some fixed values $z_{1:l}=(z_1,z_2,\dots,z_l)$ and for some measurable function $f\,:\,\mathcal{Z}^n\rightarrow \rset$,\,let $W=f(Z_1,Z_2,\dots,Z_n)$ and define for $l \in \llbracket1,n\rrbracket$:
	\begin{enumerate}
		\item $f_{l}(z_1,z_2,\dots,z_l)=\mathbb{E}\left (W \mid Z_1=z_1,Z_2=z_2,\dots,Z_l=z_l\right),$ 
		\item $\Delta_l(z_1,z_2,\dots,z_{l-1},z_l)  =f_{l}(z_1,z_2,\dots,z_{l-1},z_l)-f_{l-1}(z_1,z_2\dots,z_{l-1}),$ (the positive deviations)
		\item $D:=\max _{l=1, \ldots, n} \sup _{z_{1}, \ldots, z_{i-1} \in \mathcal{Z}} \sup _{z \in \mathcal{Z}} \Delta_{l}\left(z_{1}, \ldots, z_{l-1}, z\right),$ (the maximum positive deviation)
		\item  $\sigma_{l}^2(z_{1:l-1})=\Var\left[ \Delta_l(Z_1,Z_2,\dots,Z_{l-1},Z' ) 
		\mid Z_1=z_1, Z_2=z_2 , \dots , Z_{l-1}=z_{l-1} \right],$ where $Z'$ is an independent copy of $Z_l,$
            \item
              $\sigma^2=\sup_{z_{1:l-1}\in
              	\mathcal{Z}^{l-1}}\sum_{l=1}^n  \sigma^2_l(z_{1:l-1})$ (the
              maximum sum of variances).
	\end{enumerate}
	Then we have
	\[ \PP(W - \EE[W] > t ) \le \exp{\left(\ffrac{-t^2}{2(\sigma^2 + Dt/3)}\right)}. \]
      \end{proposition}

We also recall Proposition 2.5.2 from \cite{vershynin_2018},\,which provides an upper bound for the expectation of sub-Gaussian random variables.
\begin{proposition}\label{"Expectation-UB-subgauss"}
Let $X$ be a real valued random variable and suppose that
$$\PP(X\geq t)\leq  C_1.exp(-t^2/C_2^2),$$
for some $C_1 , C_2 > 0$.  then it holds that
\[ \EE(X) \leq M_2 C_2, \]
where  $M_2>0$ is a universal constant depending only on $C_1$.

\end{proposition}

\section{Intermediate results and detailed proofs}\label{sec:detailedProofs}

\subsection{Proof of Lemma \ref{"train-test-cond"}}\label{"proof-train-test-cond"}
Since the leave-one-out is a special case of $K$-fold  with $K=n$ (or leave-$p$-out with $p=1$)
it suffices to prove the statement concerning  the cases of the leave-$p$-out and the $K$-fold.
	
\paragraph{$K$-Fold.} For this procedure, the validation sets is a
partition of $\llbracket 1, n\rrbracket $:
\begin{equation}\label{"loo-k-fold-prop"}
  \bigcup\limits_{j=1}^K V_j=\llbracket1,n\rrbracket
  \: and \:
  V_j\bigcap V_k=\varnothing\:, \:\forall j\neq k \in \llbracket1,K\rrbracket.
\end{equation}
	
Under the assumption that $K$ divides $n$, the condition
$card(V_j)= n/K :=n_{V}$ for all the validation sets $V_j$ holds, as
stipulated in~(\ref{"card-condition"}).  Thus we
have 
\begin{equation}
  \label{"K-fold-equation"}
  n=\sum_{j=1}^{K}card(V_j)=Kn_{V}. 
\end{equation}
Furthermore,\,under \eqref{"loo-k-fold-prop"},\,any  index
$l\in \llbracket 1,n \rrbracket$ belongs to a unique
validation test $V_{j'}$ and to all the train sets $T_j=V_j^c$
with $j\neq j'$. Hence, we both have 
\begin{equation*}
  \begin{cases}
    \sum_{j=1}^{K}\indic{l \in T_j} = K-1, & \text{and}  \\
    \sum_{j=1}^{K}\indic{l \in V_j} = 1. &
  \end{cases}  
\end{equation*}
Using \eqref{"K-fold-equation"} and the fact that $n_{T}=n-n_{V}=(K-1)n_{V}$ yields
the desired result.
%
%
	
\paragraph{Leave-$p$-out.}
In the leave-$p$-out procedure, the sequence of validation sets is the
family of all subsamples $V_j$ of $\DD_n$ of size $card(V_j)=p$, thus
$K=\binom{n}{p}$. On the other hand, any index
$l\in \llbracket 1,n \rrbracket$ belongs to $\binom{n-1}{p-1}$
validation sets. Indeed constructing  a $V_j$ such as $l \in V_j$ is
equivalent to first picking $l$ and then choosing $p-1$ elements from 
$\llbracket1,n\rrbracket \setminus \{l\}$.  Hence we have	
	\begin{equation*}
	\sum_{j=1}^{K}\indic{l \in V_j} = \binom{n-1}{p-1} \: \, , \: \forall l \in \llbracket1,n\rrbracket.
	\end{equation*}
	Using the identity  $n\binom{n-1}{p-1}=p\binom{n}{p}$ we obtain 
	\[\ffrac{1}{Kn_{V}}\sum_{j=1}^{K}\indic{l \in V_j}= \ffrac{1}{p\binom{n}{p}}\sum_{j=1}^{K}\indic{l \in V_j}= 1/n. \]
	A similar argument applies to the sequence $T_{1:K}$, which completes the proof. 
	
\subsection{Intermediate results for the proofs of theorems~\ref{theo:main-sanity-rare} and~\ref{theo:main-sanity-rare-lpo}}

In this section we gather the main intermediate results involved in the proofs of our main results theorems~\ref{theo:main-sanity-rare} and~\ref{theo:main-sanity-rare-lpo}, which are of interest in their own.

A key tool to our proofs is  a Bernstein-type inequality relative to the deviation of a generic random variable $W=f(Z_1,\ldots,Z_n)$ from its mean (\cite{McDiarmid98conc}) that is recalled for convenience in section \ref{sec:genericTools} of this supplement  (Proposition~\ref{"Bernstein-extension"}). The control of the deviations involves both a maximum deviation term and a variance term.   We leverage this result to control the deviations of the  pseudo-empirical risk $\tilde \risk_\alpha$ defined in~(\ref{def:tilde-rcv}) averaged over the $K$ validation sets $V_{1:K}$. These deviations are embodied by the random variable $W$ defined in Lemma~\ref{"RCV-supp"}, Equation~\ref{def:Z-CV}, which is a key quantity when analysing the deviations of any \CV risk estimate. Controlling the deviations of $W$ is the main purpose of this section.

\begin{lemma}
	\label{"RCV-supp"}
Let $\DD_n=(Z_1,Z_2,\dots,Z_n)\in \mathcal{Z}^n$ be a sequence of random variables, and let $V_1,V_2,\dots V_K$ be validation sets that verify Assumption~\ref{assum:mask-property} with size $n_V$. Moreover, suppose that assumptions~\ref{assum:hypo-class} and~\ref{assum:cost-func} regarding the class $\mathcal{G}$ and $c$ hold. Define 
\begin{equation}\label{def:Z-CV}
W=\ffrac{1}{K}\sum_{j=1}^{K}\sup_{g \in \mathcal{G}}\lvert \Tilde{\mathcal{R}}_{\alpha}(g,V_{j})-\mathcal{R}_\alpha(g) \rvert,
\end{equation}
where $\Tilde{\mathcal{R}}_{\alpha}$ is defined by Equation~\ref{def:tilde-rcv}. Then the random variable $W$  satisfies the Bernstein-type inequality	
\[P\left(W - E(W) \geq t  \right)\leq \exp\left(\ffrac{-n\alpha t^2}{2(4+t/3)}\right).\]	
\end{lemma}

\begin{proof}
  We introduce for convenience the rescaled variable $W_\alpha = \alpha W$. Then $W_\alpha$ may be written as 
\[W_\alpha=\frac{1}{K}\sum_{j=1}^{K}\left[\ffrac{1}{n_{V}}\sup_{g\in \mathcal{G}}\Big|\sum_{i \in V_j}c(g,Z_i)\ind_{\alpha}(X_i)-\EE\left[c(g,Z)\ind_{\alpha}(X)\right ] \Big|\right],\]
where we use the shorthand notation  $\ind_{\alpha}(X)=\indic{\lVert X\rVert\geq t_\alpha }$. We derive an upper bound on $ \PP[W_\alpha - \EE(W_{\alpha})>t] $ using Proposition~\ref{"Bernstein-extension"}. Namely  we show that the maximum deviations term $D$ and $\sigma^2$ from the latter statement are respectively bounded by
$D \le 1/n$ and $\sigma^2\le 4\alpha/n$. To do so we compute explicitly   the five quantities defined in the satement of Proposition~\ref{"Bernstein-extension"} in our particular context. 

\begin{enumerate}
\item  The conditional expectations $f_l$ from Proposition~\ref{"Bernstein-extension"} are 
  (recall that $z_i=(x_i,y_i)$), 
	\begin{align*}
	&f_{l}(z_1,z_2,\dots,z_l)\\
	&=\mathbb{E}\left (W_\alpha \mid Z_1=z_1,Z_2=z_2,\dots,Z_l=z_l\right)\\
	&= \frac{1}{K n_{V}}\sum_{j=1}^{K} 
	\EE\Bigg[ \sup_{g\in \mathcal{G}}
	\Bigg|\sum_{i \in V_{j,l}}c(g,z_i)\ind_{\alpha}(x_i)+\sum_{i \in V_{j}\setminus V_{j,l} } c(g,Z_i)\ind_{\alpha}(X_i) \\
	&\hspace{4cm} -n_V\EE\left[c(g,Z)\ind_{\alpha}(X)  \right] \Bigg|\Bigg]\\
	&=\frac{1}{K n_{V}}\sum_{j=1}^{K} \EE\Bigg[ \sup_{g\in \mathcal{G}}
	\big|h_{j,g}\left(z_{1:l},Z_{l+1:n}\right)\big|\Bigg] \: ,
	\end{align*}
	where $V_{j,l}=V_j\cap \llbracket1,l\rrbracket$ are the validation indices which belong to the interval $\llbracket1,l\rrbracket$, and 
	\[h_{j,g}\left(z_{1:l},Z_{l+1:n}\right) =\sum_{i \in V_{j,l}}c(g,z_i)\ind_{\alpha}(x_i)+\sum_{i \in V_{j}\setminus V_{j,l} }c(g,Z_i) \ind_{\alpha}(X_i)-n_V\EE\left[c(g,Z)\ind_{\alpha}(X)\right] \: . \]
	\item Recall the definition of the  positive deviations $\Delta_l$, 
          \[  \Delta_l(z_1,z_2,\dots,z_{l-1},z_l)  =f_{l}(z_1,z_2,\dots,z_{l-1},z_l)-f_{l-1}(z_1,z_2\dots,z_{l-1})\: . \]
In view of the expression for $f_l$ from step 1,   we may thus  write
	\begin{align*}
	\Delta_l(z_{1:l})&=\frac{1}{K n_{V}}\sum_{j=1}^{K} \EE\Bigg[ \sup_{g\in \mathcal{G}}
	\big|h_{j,g}\left(z_{1:l},Z_{l+1:n}\right)   \big|-\sup_{g\in \mathcal{G}}
	\big|h_{j,g}\left(z_{1:l-1},Z_{l:n}\right)   \big|\Bigg] \: .
	\end{align*}
	Now, notice  that $V_{j,l}=V_{j,l-1}$ if $l \notin V_{j}$ and $V_{j,l}=V_{j,l-1}\cup\{l\}$ otherwise. Hence 
	\[
	h_{j,g}\left(z_{1:l},Z_{l+1:n}\right)-h_{j,g}\left(z_{1:l-1},Z_{l:n}\right) = \ind\{l\in V_{j}\}\bigg(c(g,z_l)\ind_{\alpha}(x_l)-c(g,Z_l)\ind_{\alpha}(X_l) \bigg) \: .\]
      Using the fact that, for any functions $f,g,$  it holds that
      $ \big\lvert\sup\lvert f \rvert - \sup \lvert g \rvert\big\rvert \leq
      \sup\lvert f-g \rvert$, we obtain 
    \begin{equation}\label{"delta_bound"}
	\lvert \Delta_l(z_{1:l}) \rvert  \leq \frac{1}{Kn_{V}}\sum_{j=1}^{K}\left[\ind\{l\in V_{j}\}\mathbb{E} \;\underbrace{\sup_{g\in \mathcal{G}}\left\lvert c(g,z_l)\ind_{\alpha}(x_l)-c(g,Z_l)\ind_{\alpha}(X_l) \right\rvert}_{(\text{Assumption \ref{assum:cost-func}})\leq 1} \right]
	\end{equation}
	and deduce that  
	\begin{align*}
	\abs{\Delta_l(z_{1:l})} &\leq \frac{1}{Kn_{V}}\sum_{j=1}^{K}\ind\{l\in V_{j}\}\\
	(\text{by Assumption \ref{assum:mask-property} }  ) &\leq \ffrac{1}{n} \: .
	\end{align*}
	\item The maximum positive deviation is defined by 
	$$D:=\max _{l=1, \ldots, n} \sup _{z_{1}, \ldots, z_{i-1} \in \mathcal{Z}} \sup _{z \in \mathcal{Z}} \Delta_{l}\left(z_{1}, \ldots, z_{l-1}, z\right). $$ 
	From the previous step, we immediately obtain \[D\leq 1/n. \]
      \item Let $Z'=(X',Y')$ be an independent copy of $Z=(X,Y)$ and
        let $z_{1:l}=(z_1,z_2,\dots,z_l)$.  Recall the conditional
        variance term from the statement of
        Proposition~\ref{"Bernstein-extension"},
        $\sigma_l^2(z_{1:l-1}):= \Var[ \Delta_l(Z_1,Z_2,\dots,Z_{l-1},Z' ) \mid Z_1=z_1, Z_2=z_2 , \dots , Z_{l-1}=z_{l-1}]$. Then $\sigma_l^2$ may be upper bounded as follows, 
	\begin{align*}
	\sigma_l^2(z_{1:l-1}) 
	&\leq \EE \left[ \Delta_l(Z_1,Z_2,\dots,Z_{l-1},Z' ) ^2 \mid Z_1=z_1, Z_2=z_2 , \dots , Z_{l-1}=z_{l-1} \right]\\
	&=\EE \left[ \Delta_l(z_1,z_2,\dots,z_{l-1},Z' ) ^2  \right] .
	\end{align*}
	Now using \eqref{"delta_bound"}, write
	\begin{align*}
          \sigma_l^2 &\leq  \ffrac{1}{(Kn_{V})^2}\mathbb{E}\left[\left(\sum_{j=1}^{K}\ind\{l\in V_{j}\}\mathbb{E}\big[\sup_{g\in \mathcal{G}}\lvert c(g,Z')\ind_{\alpha}(X')-c(g,Z_l)\ind_{\alpha}(X_l) \rvert \mid Z' \big]  \right)^2  \right]\\
          (\lvert c \rvert \leq 1)
                     &\leq (\ffrac{1}{Kn_V})^2\mathbb{E}\left[\left(\sum_{j=1}^{K}\ind\{l\in V_{j}\}\mathbb{E}\big[\sup_{g\in \mathcal{G}}\underbrace{\ind_{\alpha}(X')+\ind_{\alpha}(X_l) }_{\text{independent of $g$}} \mid Z'\big]  \right)^2 \right]\\
                     &= (\ffrac{1}{Kn_V})^2\mathbb{E}\left[\left(\sum_{j=1}^{K}\ind\{l\in V_{j}\}\underbrace{(\ind_{\alpha}(X')+\alpha )}_{\text{independent of $j$}}  \right)^2 \right]\\
                     &\leq \EE[(\ind_{\alpha}(X')+\alpha )^2]\left(\frac{1}{Kn_{V}}\sum_{j=1}^{K}\ind\{l\in V_{j}\}\right)^2 \\ & =(\alpha+3\alpha^2)\left(\frac{1}{Kn_{V}}\sum_{j=1}^{K}\ind\{l\in V_{j}\}\right)^2\\
          ( \text{by } \eqref{"key-condition"})
                     &=\ffrac{\alpha+3\alpha^2}{n^2} \\(\alpha \leq 1)
                     &\leq\ffrac{4\alpha}{n^2} \: .
	\end{align*}
	\item Finally we get $\sigma^2=\sum_{l=1}^{n}\sup_{z_{1:l-1}}\sigma_l^2(z_{1:l-1})  \leq \ffrac{4\alpha}{n}$. 
	
\end{enumerate}

At this stage, applying Proposition \ref{"Bernstein-extension"} gives 
\begin{equation*}
\PP(W_\alpha- \EE[W_\alpha] > t ) \le  \exp\Big\{\frac{-n t^2}{2(4\alpha + t/3)}\Big\} \: .
\end{equation*}

Therefore for $W=W_\alpha/\alpha$ one obtains

\begin{equation*}
\PP(W- \EE[W] > t ) \le  \exp\Big\{\frac{-n\alpha t^2}{2(4+ t/3)}\Big\} \: .
\end{equation*}

\end{proof}

To obtain a genuine probability bound on $Z$ \emph{via} the latter lemma, one also needs to control the term $E(Z)$. This is the purpose of the next lemma, the spirit of which is similar to Lemma 14 in~\cite{goix15}. The main difference \wrt  the latter reference is that we handle any bounded cost function (not only the Hamming loss), using a bound for Rademacher averages from~\cite{gine2001consistency} which applies in this broader setting.
\begin{lemma}
	\label{"RCV-expectation"}
     
        In the setting of Lemma~\ref{"RCV-supp"}, $W$ satisfies 
        \[
	\EE(W) \leq   \ffrac{M\sqrt{\vapnikG}}{\sqrt{\alpha n_V}}.
	\]

where $M>0$ is a universal constant. 
\end{lemma}

\begin{proof}
  
Notice first that, since the observations are \iid, 
	\[\mathbb{E}\left[\ffrac{1}{K}\sum_{j=1}^K\sup_{g \in \mathcal{G}}\bigg|\Tilde{\mathcal{R}}_{\alpha}(g,V_j)- \mathcal{R}_{\alpha}(g)\bigg|\right]=
	\mathbb{E}\left[\sup_{g \in \mathcal{G}}\bigg|\Tilde{\mathcal{R}}_{\alpha}(g,V_1)- \mathcal{R}_{\alpha}(g)\bigg|\right].\]
      That is, $\EE(W) = \EE(W_{V_1})$, where for a subset of indices  $S = \{1,\ldots, n_S\}$, we denote  
      \[
W_S = \sup_{g \in\mathcal{G}} \big|\Tilde \risk_\alpha(g, S) - \risk_\alpha(g)\big|.
\]
   
In order to bound $\EE[W_S]$ defined above, we follow the same steps as in the proof of Lemma~14 in \cite{goix15}, where most arguments also hold true for a bounded  VC class of cost functions.

In particular, we use the symmetrization technique. Consider Rademacher random variables $\mathcal{E}=(\epsilon_1,\epsilon_2,\dots,\epsilon_{n_S})$ taking values  in $\{-1,1\}$ and  introduce the randomized process
\[ W_{\mathcal{E}}=\sup_{g \in \mathcal{G}}\bigg|\frac{1}{\alpha n_S} \sum_{i=1}^{n_S} \epsilon_i c(g,Z_i)\un{\lVert X_i \rVert \geq t_\alpha }\bigg| \:.\]
It can be shown using the same classical steps as in the proof of Lemma 13 in \cite{goix15} that

\begin{equation*}
\label{"Rademacher-inequality"}
\EE(W_S) \leq 2\EE( W_{\mathcal{E}}) \: .
\end{equation*}
The key argument to  proceed is to condition the above expectation upon the number of indices $i$ such that $\lVert X_i \rVert \geq t_\alpha$. This conditioning trick is a standard technique for deriving non-asymptotic bounds in the \EVT framework (\cite{goix15,lhaut2021uniform}). 
Introduce a random variable $Z_{i,\alpha}$,  which has the same distribution as ($Z  \mid  \lVert X \rVert \geq t_\alpha$) and notice that 
\[ \sum_{i=1}^{n_S} \epsilon_i c(g,Z_i)\un{\lVert X_i \rVert \geq t_\alpha }\sim  \sum_{i=1}^{\mathcal{N}}\epsilon_i c(g,Z_{i,\alpha}) \:,\]
where $\mathcal{N}$ has a Binomial distribution $\mathcal{B}(n_S,\alpha)$. Equipped with these notations
\[
\EE( W_{\mathcal{E}})=
\EE(\phi(\mathcal{N})) \:,
\]

where $$\phi(N)=\EE\sup_{g \in \mathcal{G}}\bigg|\ffrac{1}{\alpha n_S} \sum_{i=1}^{N} \epsilon_i c(g,Z_{i,\alpha})\bigg|=\ffrac{N}{\alpha n_S}\EE \sup_{g \in \mathcal{G}}  \ffrac{1}{N}\bigg\lvert\sum_{i=1}^{N} \epsilon_i c(g,Z_{i,\alpha})\bigg\rvert \: .$$ 

Then by a classical Rademacher complexity arguments for finite VC-classes (see \cite{gine2001consistency}, Proposition 2.1), 
\begin{align*}
  \EE \sup_{g \in \mathcal{G}}\ffrac{1}{N} \bigg\lvert\sum_{i=1}^{N} \epsilon_i c(g,Z_{i,\alpha})\bigg\rvert& \leq \ffrac{M'_1\sqrt{\mathcal{V}_{\mathcal{G}}}}{\sqrt{N}}
                                                                                                               + \frac{M'_2 \vapnikG}{N}\; \\
& \le M'\sqrt{ \frac{\vapnikG}{N}},\end{align*}
for some universal constant $M'>0$,  whence
\begin{align*}
  \phi(N) &\leq \ffrac{N}{\alpha n_S} \ffrac{ M'\sqrt{\mathcal{V}_{\mathcal{G}}}}{\sqrt{N}}\; 
            = \ffrac{ M'\sqrt{\mathcal{V}_{\mathcal{G}} N}}{\alpha n_S}
\end{align*}
By concavity we have   $\EE(\sqrt \mathcal{N}) \le \sqrt{\EE(\mathcal{N})} = \sqrt{\alpha n_S}$, and we obtain

\[\EE( W_S)\leq 2\EE( W_{\mathcal{E}}) \leq
  \ffrac{2 M'\sqrt{\mathcal{V}_{\mathcal{G}}}}{\sqrt{n_S \alpha}}
\: .  \]
The result follows.
\end{proof}

The following probability upper bound for $Z$ follows immediately by combining Lemma~\ref{"RCV-supp"} and Lemma~\ref{"RCV-expectation"}. 
\begin{corollary}
	\label{"RCV-sup-control"}
        In the setting of Lemma~\ref{"RCV-supp"}, 

        we have
	\[P\left(W - \ffrac{M\sqrt{\vapnikG}}{\sqrt{\alpha n_V}}\geq t \right)\leq \exp\left(\ffrac{-n\alpha t^2}{2(4+t/3)}\right),\]
	where $W$ is defined in \eqref{def:Z-CV}.
      \end{corollary}

To conclude this section, we extend Theorem~10 in~\cite{goix15} bounding the supremum deviations of the empirical measure on low probability regions, to handle the case of any cost function $c$ absolutely bounded by one.
\begin{lemma}\label{lemma:inv-goix}
Recall the definitions of the risk $\risk_\alpha$ and its empirical version $\ER_\alpha$ given in Section~\ref{sec:framework} and introduce the (random) supremum deviations 
		\[W'=\sup_{g \in \mathcal{G}}\lvert \widehat{\mathcal{R}}_{\alpha}(g,S_n)-\mathcal{R}_\alpha(g) \rvert.\]
If $\mathcal{G}$ is a family of classifiers with finite VC-dimension and $c$ a bounded cost function with $\sup_{g,z}|c(g,z)|\le 1$, then,  the following Bernstein-type  inequality holds, 
	\[ \PP(W'-Q(n,\alpha) \geq t)\leq 3\exp\left(\ffrac{-n\alpha t^2}{2(4+t/3)}\right), \]
	where $Q(n,\alpha)=B(n,\alpha)+\ffrac{1}{n\alpha}$ and $B$ is defined by
        %
\begin{equation}\label{"V-n-def"}
  B(n,\alpha) = \ffrac{M\sqrt{\vapnikG}}{\sqrt{\alpha n}}
\end{equation}  
and $M$ is a universal constant. 
        
\end{lemma}

\begin{proof}
	Write $W' \leq W_1 +W_2$ with :
	\[W_1=\sup_{g \in \mathcal{G}}\lvert \widehat{\mathcal{R}}_{\alpha}(g,S_n)-\Tilde{\mathcal{R}}_{\alpha}(g,S_n) \rvert,\]
	\[ W_2=\sup_{g \in \mathcal{G}}\lvert\Tilde{\mathcal{R}}_{\alpha}(g,S_n)-\mathcal{R}_{\alpha}(g)\rvert.\]
Concerning $W_2$, applying Corollary \ref{"RCV-sup-control"} with $K=1$ , $V_1=S_n$, yields
		\begin{equation}\label{ineq:Z-1-UB}
	\PP(W_2 - B(n,\alpha) \geq t) \leq \exp\left(\ffrac{-n\alpha t^2}{2(4+t/3)}\right).
	\end{equation}
We now focus on $W_1$. Define 
$$u_{i}=\left|\un{\lVert X_{i}\rVert>\lVert X_{(\lfloor \alpha n \rfloor)}\rVert}
-\un{\lVert X_{i} \rVert \geq t_\alpha}\right|$$ 
and notice that $ W_1\leq \ffrac{1}{n\alpha}\sum_{i=1}^{n}u_{i}$. It is known (see for instance the bound for the term $A$ in \cite{jalalzai2018binary}, page 12) that 
\begin{equation*}
\ffrac{1}{n\alpha}\sum_{i=1}^{n}u_{i}\leq  \frac{1}{\alpha}\left|\frac{1}{n}\sum_{i= 1}^n \un{\lVert X\rVert \geq t_\alpha}-\alpha\right| + \frac{1}{n\alpha}. 
\end{equation*}
Now,  by noticing that $\Var(\indic{\lVert X\rVert \geq t_\alpha}) \le \alpha$ and using Bernstein's inequality, we get
\begin{align*}
\PP\left(\bigg|\frac{1}{n}\sum_{i= 1}^n \indic{\lVert X\rVert \geq t_\alpha}-\alpha \bigg|\geq t\right)&\leq 2\exp\left(\ffrac{-n t^2}{2(\alpha+t/3)}\right)\\&\leq 2\exp\left(\ffrac{-n t^2}{2(4\alpha+t/3)}\right).
\end{align*}
Finally, dividing by $\alpha$, we get 
\begin{align*}
\PP\left(\frac{1}{n\alpha}\right.&\left.\bigg|\sum_{i= 1}^n\indic{\lVert X\rVert \geq t_\alpha}-\alpha\bigg| \geq t\right)\leq 2\exp\left(\ffrac{-n\alpha t^2}{2(4+t/3)}\right).
\end{align*}
Therefore we finally obtain
\begin{equation}\label{bound:z1}
\PP\left(  \ffrac{1}{n\alpha}\sum_{i=1}^{n}u_{i} - 1/n\alpha  \geq t \right) \leq  2\exp\left(\ffrac{-n\alpha t^2}{2(4+t/3)}\right).
\end{equation}
The result follows using $W' \leq W_1 + W_2$ and $ W_1\leq \ffrac{1}{n\alpha}\sum_{i=1}^{n}u_{i}$.
\end{proof}

\subsection{Detailed proof of Theorem~\ref{theo:main-sanity-rare}}\label{sec:main-theo-proof}
In view of the argument following the statement, we derive probability upper bounds for the three terms $\DevExt$, $\DevCV$ and $\BiasCV$ defined in equations~(\ref{"dev_ext"},~\ref{"dev_CV"},~\ref{"bias_CV"}).

\paragraph{Probability bound for $\DevExt$ (see~\eqref{"dev_ext"}).} 
 Using the fact that the cost function verifies $0\leq c \leq 1$,  write 
\begin{equation*}
\hspace{-4mm}\DevExt=\abs{\widehat{ \mathcal{R}}_{CV,\alpha}(\Psi_{\alpha},V_{1:K})-\Tilde{ \mathcal{R}}_{CV,\alpha}(\Psi_{\alpha},V_{1:K})} \leq U,
\end{equation*}
where $U=\ffrac{1}{Kn_{V}\alpha}\sum_{j=1}^{K}\sum_{i\in V_j}u_{i}$ and $u_{i}=\left|\un{\lVert X_{i}\rVert>\lVert X_{(\lfloor \alpha n \rfloor)}\rVert}
-\un{\lVert X_{i} \rVert \geq t_\alpha}\right|$ as the proof of Lemma \ref{lemma:inv-goix}.
Now notice that 
\begin{align*}
U &= \ffrac{1}{Kn_{V}\alpha}\sum_{j=1}^{K}\sum_{i=1}^{n} u_{i}\indic{i \in V_j}\\
&=\ffrac{1}{Kn_{V}\alpha}\sum_{i=1}^{n}u_{i}\sum_{j=1}^{K}\un{i \in V_j}
\\&= \ffrac{1}{n\alpha}\sum_{i=1}^{n}u_{i}. \end{align*}
The last line follows from Assumption \ref{assum:mask-property}. 
Hence, using Inequality \eqref{bound:z1} from the proof of Lemma \ref{lemma:inv-goix}, we obtain
\begin{equation}\label{'c_ext'}
\PP( \DevExt -1/n\alpha  \geq t ) \leq  2\exp\left(\ffrac{-n\alpha t^2}{2(4+t/3)}\right).
\end{equation}

\paragraph{Probability bound for $\DevCV$ (see~\ref{"dev_CV"}).}
First, notice that
\begin{align*}
\DevCV&=\abs{
	\ffrac{1}{K}\sum_{j=1}^{K}\big[ \Tilde{\mathcal{R}}_{\alpha}\left(\Psi_\alpha(T_j),V_j\right)-\mathcal{R}_{\alpha}\left(\Psi_\alpha(T_j)\right) \big]} \\
&\leq \ffrac{1}{K}\sum_{j=1}^{K}\sup_{g \in \mathcal{G}}\lvert \Tilde{\mathcal{R}}_{\alpha}(g,V_j)-\mathcal{R}_{\alpha}(g) \rvert. 
\end{align*}

The \rhs of the latter display is the quantity $Z$ defined in Lemma \ref{"RCV-supp"}, Equation~\ref{def:Z-CV}. As a consequence of this lemma,
\begin{equation}\label{eq:C-CV-bound}
\PP(\DevCV-B(n_{V},\alpha)\geq t) \leq \exp\left(\ffrac{-n\alpha t^2}{2(4+t/3)}\right),  
\end{equation}
with 
\[ B(n_{V},\alpha)= \ffrac{M\sqrt{\mathcal{V}_\mathcal{G}}}{\sqrt{\alpha n_{V}}}, \]
for some universal constant $M>0$.

\paragraph{Probability bounds for $\BiasCV$ (see~\ref{"bias_CV"}).} We write
\begin{align}\label{ineq:Bias-CV-decomp}
\BiasCV&= \ffrac{1}{K}  \left|\sum_{j=1}^K\Bigg(\mathcal{R}_{\alpha}(\Psi_{\alpha}(T_j))-\mathcal{R}_{\alpha}\big(\Psi_{\alpha}(S_n)\big)\Bigg)\right| \nonumber\\
&\leq C_1+C_2,
\end{align}
with
\begin{align*}
	C_1=\ffrac{1}{K}\left|\sum_{j=1}^K \left( \mathcal{R}_{\alpha}(\Psi_{\alpha}(T_j))-\widehat{\mathcal{R}}_{\alpha}( \Psi_\alpha({T_j}) , T_j) +\widehat{\mathcal{R}}_{\alpha}( \Psi_\alpha({T_j}) , T_j)-\mathcal{R}_\alpha^*\right)\right|,
\end{align*}

\begin{align*}
	C_2=\left|\mathcal{R}_\alpha^* -  \widehat{\mathcal{R}}_{\alpha}(\Psi_{\alpha}(S_n),S_n)+\widehat{\mathcal{R}}_{\alpha}(\Psi_{\alpha}(S_n),S_n) - \mathcal{R}_{\alpha}\big(\Psi_{\alpha}(S_n)\big)\right|
\end{align*}

and \[\mathcal{R}_\alpha^*=\inf_{g \in \mathcal{G}}\mathcal{R}_\alpha(g).\]
Using the fact that  $\widehat{\mathcal{R}}_{\alpha}( \Psi_\alpha({T_j}),T_j)= \inf_{g\in \mathcal{G}}\widehat{\mathcal{R}}_{\alpha}( g, T_j)$ (Assumption~\ref{assum:ERM-settings}) and for any real functions $h$ and $f$, $\lvert \inf{h}-\inf{f}\rvert \leq  \sup\lvert h-f\rvert $ , write
\begin{align*}
\ffrac{1}{K}\left|\sum_{j=1}^K\bigg( \widehat{\mathcal{R}}_{\alpha}( \Psi_\alpha({T_j}) , T_j)-\mathcal{R}_\alpha^*\bigg)\right| &=	\ffrac{1}{K}\left|\sum_{j=1}^K \inf_{g\in \mathcal{G}}\widehat{\mathcal{R}}_{\alpha}(g, T_j)- \inf_{g\in \mathcal{G}}\mathcal{R}_{\alpha}(g)\right|\\&\leq\ffrac{1}{K}\sum_{j=1}^K\left| \inf_{g\in \mathcal{G}}\widehat{\mathcal{R}}_{\alpha}( g , T_j)- \inf_{g\in \mathcal{G}}\mathcal{R}_{\alpha}(g)\right|\\&\leq\frac{1}{K}\sum_{j=1}^K\sup_{g \in \mathcal{G}}\lvert\widehat{\mathcal{R}}_{\alpha}(g,T_j)- \mathcal{R}_{\alpha}(g)\rvert.
\end{align*}

Then, by using the triangle inequality, deduce that
\begin{align}\label{ineq:C-1-decomp}
C_1
&\leq  \ffrac{2}{K}\sum_{j=1}^K\sup_{g \in \mathcal{G}}\lvert\widehat{\mathcal{R}}_{\alpha}( g,T_j)- \mathcal{R}_{\alpha}(g)\rvert \nonumber\\
&\leq 2(W_1+W_2),
\end{align}

with  
\[W_1=\ffrac{1}{K}\sum_{j=1}^K\sup_{g \in \mathcal{G}}\lvert
\widehat{\mathcal{R}}_{\alpha}( g , T_j)-
 \Tilde{\mathcal{R}}_{\alpha}(g,T_j)\rvert,\] 
\[W_2=\ffrac{1}{K}\sum_{j=1}^K\sup_{g \in \mathcal{G}}\lvert\Tilde{\mathcal{R}}_{\alpha}(g,T_j)- \mathcal{R}_{\alpha}(g)\rvert.\]
In order to bound $Z_2$ we use the fact that the statements of   Lemmas~\ref{"RCV-supp"}, \ref{"RCV-expectation"} and Corollary~\ref{"RCV-sup-control"} still hold true if one substitutes the training sets $T_j$ for the validation sets $V_j$, and the training size $n_T$ for $n_V$, as revealed by an inspection of the proofs. 
The only difference between the two arguments is in steps 2 and 4 from the proof of Lemma~\ref{"RCV-supp"} where we use the identity $	\frac{1}{K}\sum_{j=1}^{K}\ffrac{\indic{l \in T_j} }{n_{T}} 
        =\frac{1}{n} 
$ 
from Lemma~\ref{"train-test-cond"} instead of Identity~(\ref{"key-condition"}).  
Thus we obtain, from the twin statement of Corollary~\ref{"RCV-sup-control"}, 
\begin{equation}\label{ineq:BCV-Z-2}
	\PP(W_2-2B(n_{T},\alpha) \geq 2t ) \leq 2\exp\left(\ffrac{-n\alpha t^2}{2(4+t/3)}\right), 
\end{equation}
where $B(n,\alpha)$ is defined in~(\ref{"V-n-def"}).
Similarly the term $W_1$ can be bounded following the same argument as in the first paragraph (Probability bound for $\DevExt$) up to  replacing $V_j$ with $T_j$ and $n_V$ with $n_T$ :
\begin{align*}
	W_1&\leq  \ffrac{1}{Kn_{T}\alpha}\sum_{j=1}^{K}\sum_{i=1}^{n} u_{i}\indic{i \in T_j}\\
	&=\ffrac{1}{Kn_{T}\alpha}\sum_{i=1}^{n}u_{i}\sum_{j=1}^{K}\un{i \in T_j}
	\\&= \ffrac{1}{n\alpha}\sum_{i=1}^{n}u_{i}, 
\end{align*}
where the last line follow from the claim after Assumption \ref{assum:mask-property} ($\frac{1}{K}\sum_{j=1}^{K} {\indic{l \in T_j} }=\frac{n_T}{n}$). This yields using similar arguments as before
\begin{equation*}
\PP(W_1 - \frac{1}{n\alpha} \geq 2t ) \leq 4\exp\left(\ffrac{-n\alpha t^2}{2(4+t/3)}\right). 
\end{equation*}
Using the fact that $n_T\leq n$ it follows
\begin{equation}\label{ineq:BCV-Z-1}
\PP(W_1 - \frac{1}{n_T\alpha} \geq 2t ) \leq 4\exp\left(\ffrac{-n\alpha t^2}{2(4+t/3)}\right). 
\end{equation}
Decomposition (\ref{ineq:C-1-decomp}) combined with inequalities (\ref{ineq:BCV-Z-2}), (\ref{ineq:BCV-Z-1}) leads to 
\begin{equation}\label{eq:first-term-bound}
\PP\left(C_{1}-2Q(n_{T},\alpha) \geq   4t \right) \leq  6\exp\left(\frac{-n\alpha t^2}{2(4+t/3)}\right), 
\end{equation}
with 
\begin{align}
	Q(n,\alpha)&=B(n,\alpha)+\ffrac{1}{n\alpha}\nonumber\\ &=\ffrac{M\sqrt{\mathcal{V}_\mathcal{G}}}{\sqrt{\alpha n}}+\ffrac{1}{n\alpha}. \label{eq:Q-def}
\end{align}
Using the same technique, one has
\[C_2 \leq 2\sup_{g \in \mathcal{G}}\lvert\widehat{\mathcal{R}}_{\alpha}( g,S_n)- \mathcal{R}_{\alpha}(g)\rvert.\]
Then by Lemma~\ref{lemma:inv-goix}, we obtain
\begin{align}\label{eq:second-term-bound}
\PP\left(C_{2}-2Q(n,\alpha) \geq 4t \right) \leq  6\exp\left(\frac{-n\alpha t^2}{2(4+t/3)}\right). 
\end{align}
Moreover, notice that we have 
\begin{equation*}
\begin{cases}
4B(n_{T},\alpha) \geq 2B(n,\alpha) +2B(n_{T},\alpha), \\
\ffrac{4}{n_{T}\alpha}     \geq \ffrac{2}{n_{T}\alpha} + \ffrac{2}{n\alpha}.
\end{cases}
\end{equation*}
Therefore we get
\begin{equation}\label{eq:monotone-ineq}
	4Q(n_{T},\alpha) \geq 2Q(n,\alpha)+2Q(n_{T},\alpha).
\end{equation}
Combining equations (\ref{eq:first-term-bound}), (\ref{eq:second-term-bound}), (\ref{eq:monotone-ineq}) and decomposition (\ref{ineq:Bias-CV-decomp}) yields
\begin{equation}\label{"C_train_bound"}
\PP\left(\BiasCV-4Q(n_{T},\alpha)\geq 8t \right) \leq  12\exp\left(\ffrac{-n\alpha t^2}{2(4+t/3)}\right).  
\end{equation}
\paragraph{Assembling terms.}
Using equations (\ref{'c_ext'}), (\ref{eq:C-CV-bound}), (\ref{"C_train_bound"}) and the decomposition (\ref{ineq:main-decomposition}), deduce the inequality
\begin{equation}\label{ineq:CV-Base-UB}
\PP\bigg(\big|\widehat{\mathcal{R}}_{CV,\alpha}(\Psi_{\alpha},V_{1:K})-\mathcal{R}_{\alpha}\big(\Psi_{\alpha}(S_n)\big)\big| - E_{CV}(n_{T},n_{V},\alpha)\geq 10t \bigg)  \leq  15\exp\left(\ffrac{-n\alpha t^2}{2(4+t/3)}\right),  
\end{equation}
with 
\begin{align*}
&E_{CV}(n_{T},n_{V},\alpha)=B(n_{V},\alpha)+4Q(n_{T},\alpha)+\ffrac{1}{n_{T}\alpha} \nonumber  \\ &=M\sqrt{\mathcal{V}_\mathcal{G}}(\frac{1}{\sqrt{n_{V}\alpha}}+\frac{4}{\sqrt{n_{T}\alpha}})+\ffrac{5}{n_{T}\alpha}.
\end{align*}

The last line follows, using the definitions of $B$ in  \eqref{"V-n-def"}) and $Q$  in \eqref{eq:Q-def}).

By inverting inequality (\ref{ineq:CV-Base-UB}), one has, with probability $1-15\delta$,
\begin{equation*}
\bigg|\widehat{\mathcal{R}}_{CV,\alpha}(\Psi_{\alpha},V_{1:K})-\mathcal{R}_{\alpha}\big(\Psi_{\alpha}(S_n)\big) \bigg| \leq   E_{CV}(n_{T},n_{V},\alpha)+ \frac{20}{3 n \alpha}\log(\frac{1}{\delta}) + 20\sqrt{\frac{2}{n\alpha}\log(\frac{1}{\delta})},
\end{equation*}
which is the desired result. 

\subsection{Intermediate results for the proof of Theorem~\ref{theo:main-sanity-rare-lpo}}\label{sec:technicalResults}

\begin{lemma}\label{"RCV-ge-RN"}
Let $\Psi_{\alpha}$ be the ERM rule on the tail region of level $1-\alpha$ defined in (\ref{"g_alpha_j_def"}). Given a dataset	$\DD_n=(Z_1,Z_2,\dots,Z_n)\in \mathcal{Z}^n$  it holds that
\[   
\widehat{ \mathcal{R}}_{\alpha}(\Psi_{\alpha}(S_n),S_n) \leq  \widehat{\mathcal{R}}_{CV,\alpha}(\Psi_{\alpha},V_{1:K}).
\]
In other words the \CV risk estimate of the ERM rule cannot be less than the empirical risk evaluated on the full dataset. 
\end{lemma}

\begin{proof}
The argument for $\alpha\neq 1$ is the same as the one for
 $\alpha=1$ (standard ERM) which may be found in \cite{kearns1999algorithmic}. We reproduce
  it for the sake of completeness. By definition of
  $\Psi_{\alpha}(S_n)$, one has

\begin{equation*}
  \forall j \in \llbracket 1,n \rrbracket \:,\: \widehat{\mathcal{R}}_{\alpha}(\Psi_{\alpha}(T_j),S_n) \geq \widehat{\mathcal{R}}_{\alpha}(\Psi_{\alpha}(S_n),S_n),
\end{equation*}

since $\widehat{\mathcal{R}}_{\alpha}(g,S_n)= \ffrac{1}{n}\left( n_{V}\widehat{\mathcal{R}}_{\alpha}(g,V_j)+n_{T}\widehat{\mathcal{R}}_{\alpha}(g,T_j) \right)$, for $g\in \mathcal{G}$. It follows that
\begin{align*}
	\ffrac{1}{n}\left( n_{V}\underbrace{\widehat{\mathcal{R}}_{\alpha}(\Psi_\alpha(T_j),V_j)}_{\text{validation error}}+n_{T}\underbrace{\widehat{\mathcal{R}}_{\alpha}(\Psi_{\alpha}(T_j),T_j)}_{\text{training error}}  \right) \geq \ffrac{1}{n}\big( n_{V}\widehat{\mathcal{R}}_{\alpha}(&\Psi_{\alpha}(S_n),V_j)\\& +  n_{T}\widehat{\mathcal{R}}_{\alpha}(\Psi_{\alpha}(S_n),T_j) \big).
\end{align*}

Since $\Psi_{\alpha}(T_j)$ minimizes the training error on the $j$'th training set $T_j$, in particular we have $$\widehat{\mathcal{R}}_{\alpha}(\Psi_{\alpha}(S_n),T_j) \geq  \widehat{\mathcal{R}}_{\alpha}(\Psi_{\alpha}(T_j),T_j),$$ hence 
\begin{equation}\label{ineq:cv-er-base}
	\widehat{\mathcal{R}}_{\alpha}(\Psi_{\alpha}(S_n),V_j) \leq  \widehat{\mathcal{R}}_{\alpha}(\Psi_{\alpha}(T_j),V_j)\:,\:\forall j \in \llbracket 1 , K \rrbracket.
\end{equation}

In addition the average empirical risk of $\Psi_\alpha(S_n)$ is equal to the empirical risk on the full dataset, indeed
\begin{align*}
	\ffrac{1}{K}\bigg(\sum_{j=1}^{K}\widehat{\mathcal{R}}_{\alpha}(\Psi_{\alpha}(S_n),V_j)\bigg)&=\ffrac{1}{Kn_{V}}\bigg(\sum_{j=1}^{K}\sum_{i\in V_j} c\big(\Psi_{\alpha}(S_n),O_i)\un{\lVert X_{i}\rVert>\lVert X_{(\lfloor \alpha n \rfloor\big)}\rVert }\bigg)\\
	&\hspace{-8mm}=\ffrac{1}{Kn_{V}}\bigg(\sum_{j=1}^{K}\sum_{i=1}^{n}c\big(\Psi_{\alpha}(S_n),O_i\big)\un{i \in V_j}\un{\lVert X_{i}\rVert>\lVert X_{(\lfloor \alpha n \rfloor)}\rVert }\bigg)\\
	&\hspace{-8mm}=\ffrac{1}{Kn_{V}}\bigg(\sum_{i=1}^{n}c\big(\Psi_{\alpha}(S_n),O_i\big)\un{\lVert X_{i}\rVert>\lVert X_{(\lfloor \alpha n \rfloor)}\rVert }\bigg)\sum_{j=1}^{K}\un{i \in V_j}\\
	(\text{By Assumption }\ref{assum:mask-property})&=\widehat{\mathcal{R}}_{\alpha}\left(\Psi_{\alpha}(S_n),S_n\right).
\end{align*}

Thus  by averaging Inequality \eqref{ineq:cv-er-base}, we get
\[   
\widehat{ \mathcal{R}}_{\alpha}(\Psi_{\alpha}(S_n),S_n) \leq  \widehat{\mathcal{R}}_{CV,\alpha}(\Psi_{\alpha},V_{1 : K}).
\]
\end{proof}
The following lemma is used in the proof of our second main result concerning \lpo risk estimation, see Inequality~(\ref{eq:keyTheorem2fromLemma}). It is a generalization of  Markov inequality that is particularly useful for  cross-validation estimates. Our proof shares similarities with the proof of Theorem~4.1 in~\cite{kearns1999algorithmic} formulated under general algorithmic stability assumptions. 
\begin{lemma}\label{"RCV-markov"}
In the setting of Theorem \ref{theo:main-sanity-rare}, we have 
\begin{align*}
	\PP(\rcvEx(\Psi_{\alpha},V_{1 : K}) - \mathcal{R}_{CV,\alpha}(\Psi_{\alpha},V_{1 : K})  \geq &t ) \leq \\
	&\hspace{-10mm}\ffrac{\EE(\DevExt+\BiasCV+\abs{ \hat\risk_\alpha(\Psi_{\alpha}(S_n),S_n) - \risk_\alpha(\Psi_{\alpha}(S_n)) })}{t},
\end{align*}

where  $\BiasCV$ (resp $\DevExt$) is defined by Equation~\ref{"bias_CV"} (\emph{resp.} Equation~\ref{"dev_ext"}).
\end{lemma}

\begin{proof}
Set 
$\rcvEx=\rcvEx(\Psi_{\alpha},V_{1 : K})$, $\mathcal{R}_{CV,\alpha}=\mathcal{R}_{CV,\alpha}(\Psi_{\alpha},V_{1 : K})$ , $  \hat\risk_\alpha=\hat\risk_\alpha(\Psi_{\alpha}(S_n),S_n)$, $ \risk_\alpha=\risk_\alpha(\Psi_{\alpha}(S_n))$.\\
 For any integrable real valued random variable $L$, and any $t>0$ write  
  $$\EE[L]= \PP[L\ge t]\EE[L \,|\, L\ge t] + \EE[L \un{L<t}].$$ Reorganising, we obtain the following generalized Markov inequality, 
  \[
    \PP[L\geq t] = \frac{\EE[L] - \EE[L \un{L<t}]}{\EE[L \,|\, L\ge t]}
    \le \frac{\EE(L) - \EE(L \un{L<t})}{t}. 
  \]
  Letting $L = \rcvEx-\mathcal{R}_{CV,\alpha}$ we obtain
\begin{align}\label{"exp-CV-one"}
\PP(\rcvEx-\mathcal{R}_{CV,\alpha}&\geq t)=\ffrac{\EE(\rcvEx-\mathcal{R}_{CV,\alpha})}{t}\nonumber\\&\hspace{5em} \frac{-\EE\left[(\rcvEx-\mathcal{R}_{CV,\alpha}) \un{\rcvEx-\mathcal{R}_{CV,\alpha} \leq t} \right]}{t}.
\end{align}

Using the fact that $\EE(\tilde{\mathcal{R}}_{CV,\alpha}-\risk_{CV,\alpha})=0$ and that $\DevExt=\abs{\rcvEx-\tilde{\mathcal{R}}_{CV,\alpha}}$, one gets
\begin{align}
\EE(\rcvEx-\mathcal{R}_{CV,\alpha})&=\EE(\rcvEx-\tilde{\mathcal{R}}_{CV,\alpha}) \nonumber \\ &\leq \EE(\DevExt).
\label{"exp-CV-two"}
\end{align} 
Now using lemma  \ref{"RCV-ge-RN"} write
\begin{align}
\EE\left[(\mathcal{R}_{CV,\alpha}-\rcvEx) \ind_{\rcvEx-\mathcal{R}_{CV,\alpha} \leq t} \right] &\leq \EE\left[(\mathcal{R}_{CV,\alpha}-\widehat\risk_\alpha) \ind_{\rcvEx-\mathcal{R}_{CV,\alpha} \leq t} \right] \nonumber \\ & \leq \EE\left[\abs{\mathcal{R}_{CV,\alpha}-\widehat\risk_\alpha} \ind_{\rcvEx-\mathcal{R}_{CV,\alpha} \leq t} \right]  \nonumber \\ & 
\leq  \EE\left[\abs{\mathcal{R}_{CV,\alpha}-\widehat\risk_\alpha} \right] \nonumber \\ &
\leq \EE\left[\abs{\mathcal{R}_{CV,\alpha}-\risk_\alpha} \right]+ \EE\left[\abs{\risk_\alpha-\widehat\risk_\alpha} \right] \nonumber \\
&=\EE\left[\BiasCV \right]+ \EE\left[\abs{\risk_\alpha-\widehat\risk_\alpha} \right].
\label{"exp-CV-three"}
\end{align}

Where the $\BiasCV$ term in the last line is defined in \eqref{"bias_CV"}. Combining Inequality~\ref{"exp-CV-one"} with equations~\ref{"exp-CV-two"} and~\ref{"exp-CV-three"} yields
	\[\PP(\rcvEx-\mathcal{R}_{CV,\alpha} \geq t ) \leq \ffrac{\EE(\DevExt+\BiasCV+\abs{\hat\risk_\alpha- \risk_{\alpha}})}{t}, \]
which concludes the proof.
\end{proof}

\subsection{Proof of Theorem~\ref{theo:main-sanity-rare-lpo}}\label{"main-theo-lpo-proof"}

In view of the discussion following the statement of the theorem (namely the risk decomposition~(\ref{"basic-inequality"}) and the Markov-type inequality~(\ref{eq:keyTheorem2fromLemma})) and the bound for the term $\BiasCV$ obtained in \eqref{"C_train_bound"}, we only need to obtain bounds for the expectations $\EE\big(\abs{\hat\risk_\alpha(\Psi_\alpha(S_n),S_n)- \risk_{\alpha}(\Psi_\alpha(S_n))}\big)$, $\EE(\DevExt)$,
and  $\EE(\BiasCV)$  .
The proof will then be completed by combining together the different terms. 

	\paragraph{Bounding $\mathbb{E}(\DevExt)$.} By Equation~\ref{'c_ext'}, one has 
	\begin{equation*}
	\PP( \DevExt -\frac{1}{n\alpha} \geq t ) \leq  \exp\left(\ffrac{-n\alpha t^2}{2(4+t/3)}\right) \: .
	\end{equation*}
On the one hand, under Assumption~\ref{assum:cost-func} one has $\PP(\DevExt -\frac{1}{n\alpha}  \geq t) =0$ for $t \geq 2$. On the other hand, the following inequality holds,
	\[\forall t \leq 2 \: , \: 2(4+t/3) \leq 10 \: .\]  
	Hence we may write, for $t\geq0$, 
	\begin{equation}\label{ineq:C-ext-larger-bound}
	 \PP( \DevExt -\frac{1}{n\alpha}   \geq t ) \leq  \exp\left(\ffrac{-n\alpha t^2}{10}\right) \: .
	\end{equation}
	Therefore by Proposition~\ref{"Expectation-UB-subgauss"}, we get
	\begin{align}\label{"exp_c_ext"}
	\EE(\DevExt) &\leq \frac{1}{n\alpha} + \frac{M_1}{\sqrt{n\alpha}} \nonumber\\ & \leq \frac{1}{n_{T}\alpha} + \frac{M_1}{\sqrt{n_{T}\alpha}} \:,
	\end{align}
	for some universal constant $M_1  > 0$.
	
	\paragraph{Bounding $\mathbb{E}(\BiasCV)$.} 
	Using Equation~\ref{"C_train_bound"} and reasoning as in the previous paragraph  leads to
	\begin{align}\label{"exp_c_train"}
	\EE(\BiasCV) \leq    4Q(n_{T},\alpha)+\frac{M_2}{\sqrt{n_{T}\alpha}},
	\end{align}
	where $Q(n,\alpha)$ is defined by \eqref{eq:Q-def} and $M_2>0$ is a universal constant, independent of $\mathcal{G}$,\,$n$ and $\alpha$.
	\paragraph{Bounding $\EE(\abs{\hat\risk_\alpha(\Psi_{\alpha}(S_n),S_n)- \risk_{\alpha}(\Psi_{\alpha}(S_n))})$.} By Lemma~\ref{lemma:inv-goix} we obtain 
	\begin{align}\label{"risk_alpha_bound"}	
	\PP( \big|\hat\risk_\alpha(\Psi_{\alpha}(S_n),S_n)- \risk_{\alpha}(\Psi_\alpha(S_n))\big|  - Q(n,\alpha)  \geq t )  \leq 3\exp\left(\frac{-n\alpha t^2}{2(4+t/3)}\right)\:.
\end{align}
	Then we get
	\begin{align}\label{"exp_c_ER"}
	\EE(\big|\hat\risk_\alpha(\Psi_{\alpha}(S_n),&S_n)- \risk_{\alpha}(\Psi_{\alpha}(S_n))\big|) \nonumber\\&\leq Q(n,\alpha) +\frac{M_3}{\sqrt{n\alpha}} \nonumber 
	\\& \leq  Q(n_{T},\alpha) +\frac{M_3}{\sqrt{n_{T}\alpha}} \:,
	\end{align}
for some universal constant $M_3>0$. 

Combining equations \eqref{"exp_c_ext"}, \eqref{"exp_c_train"}, \eqref{"exp_c_ER"} with Lemma \ref{"RCV-markov"} gives 
\begin{equation}\label{"Markov-CV-ineq"}
\PP(\widehat{ \mathcal{R}}_{CV,\alpha}(\Psi_{\alpha},V_{1:K})-\mathcal{R}_{CV,\alpha}(\Psi_{\alpha},V_{1:K}) \geq t ) \leq  \frac{5Q(n_{T},\alpha) +\ffrac{M_4}{\sqrt{n_{T}\alpha}}}{t} 
+ \ffrac{1/(n_{T}\alpha)}{t}\:,
\end{equation} 
where $M_4=M_1+M_2+M_3$. The next step is to derive a probability bound for $\mathcal{R}_{CV,\alpha}(\Psi_{\alpha},V_{1:K})-\widehat{ \mathcal{R}}_{CV,\alpha}(\Psi_{\alpha},V_{1:K})$. We have
\begin{align}
\PP(\mathcal{R}_{CV,\alpha}(\Psi_{\alpha},V_{1:K})-\widehat{ \mathcal{R}}_{CV,\alpha}(&\Psi_{\alpha},V_{1:K})
-5Q(n_{T},\alpha) \geq 9t )  \nonumber\\
&\hspace{-8mm}\leq	\PP(\mathcal{R}_{CV,\alpha}(\Psi_{\alpha},V_{1:K})-\widehat{ \mathcal{R}}_{\alpha}(\Psi_{\alpha}(S_n),S_n)
-5Q(n_{T},\alpha) \geq 9t )	\nonumber \\
&\hspace{-8mm}\leq  \PP(\mathcal{R}_{CV,\alpha}(\Psi_{\alpha},V_{1:K})-\risk_{\alpha}(\Psi_{\alpha}(S_n))-4Q(n_{T},\alpha)\geq 8t )\nonumber\\&\hspace{3mm} + \PP(\risk_{\alpha}(\Psi_{\alpha}(S_n))-\widehat\risk_\alpha(\Psi_{\alpha}(S_n),S_n) \nonumber-Q(n_{T},\alpha)\geq t)	\nonumber \\
&\hspace{-8mm}\leq \PP(\big|\widehat\risk_\alpha(\Psi_{\alpha}(S_n),S_n))- \risk_{\alpha}(\Psi_{\alpha}(S_n))\big|-Q(n_{T},\alpha) \geq t) \nonumber  \\& \hspace{5mm}+\PP(\BiasCV -4Q(n_{T},\alpha)\geq 8t )\nonumber\\
(\text{By  } \eqref{"C_train_bound"} + \eqref{"risk_alpha_bound"})&\leq 15\exp\left(\ffrac{-n\alpha t^2}{2(4+t/3)}\right) . \label{"neg-Markov-CV-ineq"}
\end{align}
The first inequality follows from the fact that $\widehat{ \mathcal{R}}_{CV,\alpha}(\Psi_{\alpha},V_{1:K}) \geq  \widehat{ \mathcal{R}}_{\alpha}(\Psi_{\alpha}(S_n),S_n)$ (lemma \ref{"RCV-ge-RN"}). The second inequality is obtained by a union bound. The third inequality follows from the definition of $\BiasCV$ in \eqref{"bias_CV"}. Combining \eqref{"Markov-CV-ineq"}, \eqref{"neg-Markov-CV-ineq"} and that  
\begin{align*}
\PP(\abs{X}-5Q(n_{T},\alpha)\geq 9t) &\leq \PP(X -5Q(n_{T},\alpha)\geq 9t)+ \PP(-X-5Q(n_{T},\alpha)\geq 9t)\\ &\leq    \PP(X\geq 9t)+ \PP(-X-5Q(n_{T},\alpha)\geq 9t),
\end{align*}
leads to 
\begin{align}\label{ineq: markov-cv-dev}
\nonumber &\PP\bigg(\big|\widehat{ \mathcal{R}}_{CV,\alpha}(\Psi_{\alpha},V_{1:K})-\mathcal{R}_{CV,\alpha}(\Psi_{\alpha},V_{1:K}) \big| -5Q(n_{T},\alpha) \geq 9t \bigg)  \\
&\leq  \frac{5Q(n_{T},\alpha) }{t} +\frac{M_4/\sqrt{n_{T}\alpha} }{t} + \ffrac{(1/n_{T}\alpha)}{t}  + 15\exp\left(\ffrac{-n\alpha t^2}{2(4+t/3)}\right) \: .
\end{align}
Finally, using \eqref{"basic-inequality"}, we get
\begin{align}\label{ineq:UB-MARkov-LPO}
\nonumber &\PP\Big(\big|\widehat{ \mathcal{R}}_{CV,\alpha}(\Psi_{\alpha},V_{1:K})-\mathcal{R}_{\alpha}\big(\Psi_{\alpha}(S_n)\big)\big| -9Q(n_{T},\alpha) \geq 17t\Big)\\
\nonumber &\leq\PP(\BiasCV-4Q(n_{T},\alpha)\geq 8t)  \\&\hspace{3mm}+\PP\Big(\big|\widehat{ \mathcal{R}}_{CV,\alpha}(\Psi_{\alpha},V_{1:K})-\mathcal{R}_{CV,\alpha}(\Psi_{\alpha},V_{1:K})\big|-5Q(n_{T},\alpha)\geq 9t\Big) \\
& \leq \frac{5Q(n_{T},\alpha) +(M_4/\sqrt{n_{T}\alpha})}{t}
+\ffrac{1/(n_{T}\alpha)}{t}+ 27\exp\left(\ffrac{-n\alpha t^2}{2(4+t/3)}\right) .
\end{align}
The last line follows from \eqref{"C_train_bound"} and \eqref{ineq: markov-cv-dev}. Since for any $t\geq 2$,
\begin{equation*}
\PP\Big(\big|\widehat{ \mathcal{R}}_{CV,\alpha}(\Psi_{\alpha},V_{1:K})-\mathcal{R}_{\alpha}\big(\Psi_{\alpha}(S_n)\big)\big| -9Q(n_{T},\alpha) \geq 17t\Big)=0 
\end{equation*}
we can restrict our attention to the case $t\leq 2$, for which  we have
\[ 27\exp\left(\ffrac{-n\alpha t^2}{2(4+t/3)}\right) \leq 27\exp\left(\ffrac{-n\alpha t^2}{10}\right) . \]
Using that $\exp(-x)\leq \ffrac{1}{x}$ for $x \geq 0$, we deduce that
\begin{align*}
27\exp\left(\ffrac{-n\alpha t^2}{2(4+t/3)}\right)  \leq \frac{270}{n\alpha t^2}.
\end{align*}
Using the latter inequality and inverting \eqref{ineq:UB-MARkov-LPO}, we get that with probability $1-17\delta$,
\begin{align*}
\big|\widehat{ \mathcal{R}}_{CV,\alpha}(\Psi_{\alpha},V_{1:K})-\mathcal{R}_{\alpha}\big(\Psi_{\alpha}(S_n)\big)\big| \leq 9Q(n_{T},\alpha) +&\frac{5Q(n_{T},\alpha) +  (M_4/\sqrt{n_{T}\alpha})+(1/n_{T}\alpha)}{\delta}\\&\hspace{12mm} +\sqrt{\frac{270}{n_{T}\alpha\delta}},
\end{align*}
Using the fact that  $ \sqrt{\ffrac{1}{\delta}}\leq\ffrac{1}{\delta}$ (since $\delta \leq 1$), the latter inequality becomes:
\begin{align*}
|\widehat{ \mathcal{R}}_{CV,\alpha}(\Psi_{\alpha},V_{1:K})-\mathcal{R}_{\alpha}\big(\Psi_{\alpha}(S_n)\big) | \leq 
9Q(n_{T},\alpha)+
\ffrac{1}{\delta\sqrt{n_{T}\alpha}}(5Q(n_{T},\alpha)+M_5)+\ffrac{1}{\delta n_{T}\alpha},
\end{align*}
with $M_5=M_4+\sqrt{270}$. Replacing $Q$ with its expression (Equation~\ref{eq:Q-def}) gives the desired result.



\end{document}